\documentclass[12pt,reqno]{amsart}
\usepackage{amsfonts,amsmath}
\usepackage{amssymb}
\usepackage{textcomp}
\usepackage{enumitem}
\usepackage{graphicx,color}
\usepackage{setspace}
\numberwithin{equation}{section}
\allowdisplaybreaks

\newcommand{\nc}{\newcommand}
\nc{\parent}[1]{$[\![#1]\!]$}
\newtheorem{theorem}{Theorem}[section]
\newtheorem{lemma}{Lemma}[section]
\newtheorem{example}{Example}[section]
\newtheorem{corollary}{Corollary}[section]
\newtheorem{proposition}{Proposition}[section]
\newtheorem{remark}{Remark}[section]
\newtheorem{definition}{Definition}[section]

\newcommand{\Implies}[2]{$\text{\ref{#1}}\implies\text{\ref{#2}}$}
\newenvironment{pf-main}{{\sc Proof of Theorem \ref{mainresult}.}\hspace{3mm}}{\qed}

\nc{\cadlag}{c\`{a}dl\`{a}g } \nc{\Ito}{It\^o} \nc{\ba}{\begin{array}}\nc{\elr}{(\ell,r)}
	\nc{\ea}{\end{array}} \nc{\be}{\begin{equation}}
	\nc{\ee}{\end{equation}} \nc{\bea}{\begin{eqnarray}}
	\nc{\eea}{\end{eqnarray}} \nc{\bean}{\begin{eqnarray*}}
	\nc{\eean}{\end{eqnarray*}} \nc{\bu}{\bullet} \nc{\nn}{\nonumber}
\nc{\cA}{{\mathcal A}} \nc{\cB}{{\mathcal B}} \nc{\cE}{\mathcal{E}}\nc{\cC}{{\mathcal
		C}} \nc{\cD}{{\mathcal D}} \nc{\bbD}{\mathbb{D}}
\nc{\cG}{{\mathcal G}} \nc{\cF}{{\mathcal F}} \nc{\cS}{{\mathcal
		S}} \nc{\cU}{{\mathcal U}} \nc{\cH}{{\mathcal H}}
\nc{\cK}{{\mathcal K}}\nc{\cL}{{\mathcal L}}  \nc{\cM}{{\mathcal
		M}} \nc{\cO}{{\mathcal O}} \nc{\cP}{{\mathcal P}} \nc{\bfE}{\mathbf{E}}
\nc{\bbE}{\mathbb{E}} \nc{\tbA}{\tilde{\bbA}}\nc{\bbA}{\mathbb{A}}\nc{\bbF}{\mathbb{F}}
\nc{\bbEQ}{\mathbb{E}_{\mathbb{Q}}} \nc{\eps}{\varepsilon}
\nc{\bbEP}{\mathbb{E}_{\mathbb{P}}}\nc{\bbL}{\mathbb{L}}
\nc{\what}{\widehat} \nc{\bbP}{\mathbb{P}} \nc{\bbQ}{\mathbb{Q}}
\nc{\del}{\partial} \nc{\Om}{\Omega} \nc{\om}{\omega}
\nc{\bbR}{\mathbb{R}} \nc{\bbN}{\mathbb{N}} \nc{\fps}{$(\Om, \cF,
	(\cF_t)_{t\geq 0}, \bbP)$} \nc{\bbC}{\mathbb{C}}
\nc{\bfr}{\begin{flushright}} \nc{\efr}{\end{flushright}}
\nc{\dXt}{\Delta X_{t}} \nc{\dXs}{\Delta X_{s}}
\nc{\bs}{\blacksquare} \nc{\dX}{\Delta X} \nc{\dY}{\Delta Y}
\nc{\dnkx}{\left(X(T^{n}_{k})-X(T^{n}_{k-1})\right)}
\nc{\esssup}{\mathrm{ess}\mbox{ }\mathrm{sup}}
\nc{\essinf}{\mathrm{ess}\mbox{ } \mathrm{inf}}
\nc{\dhats}{\widehat{\delta_s}} \nc{\half} {\frac{1}{2}}

\nc{\ol}{\overline}
\def\rar{\rightarrow}

\nc{\chf}{\mbox{$\mathbf1$}}

\begin{document}
	
	\title[Integral representation of defective functions]{Minimal subharmonic functions and related integral representations}
	\author{Umut \c{C}et\.in}
	\address{Department of Statistics, London School of Economics and Political Science, 10 Houghton st, London, WC2A 2AE, UK}
	\email{u.cetin@lse.ac.uk}
	\date{\today}
	\begin{abstract}
		A Choquet-type integral representation result for non-negative subharmonic functions of a one-dimensional regular diffusion is established. The representation allows in particular an integral equation for strictly positive subharmonic functions that is driven by the Revuz measure of the associated continuous additive functional.    Moreover, via the aforementioned integral equation,  one can construct an {\em \Ito-Watanabe pair} $(g,A)$ that consist of a subharmonic function $g$ and  a continuous additive functional $A$ is with Revuz measure $\mu_A$ such that $g(X)\exp(-A)$ is a local martingale. Changes of measures associated with \Ito-Watanabe pairs are studied and shown to modify the long term behaviour of the original diffusion process to exhibit transience. 
	\end{abstract}
	\maketitle

\section{Introduction}
One of the fundamental results in the potential theory of Markov processes is the Riesz representation of an excessive (non-negative superharmonic) function as the sum of a harmonic function and the potential of a measure (see, e.g., Section VI.2 in \cite{BG}, \cite{Duncan71} and \cite{ChungRao80} for proofs under various assumptions). In the particular setting of a regular transient one-dimensional diffusion this amounts to a finite excessive function $f$ having the following representation:
\[
f(x)=\int u(x,y)\mu(dy)+ h(x),
\]
where $h$ is a harmonic function, $u$ is the potential density describing the {\em minimal} excessive functions, and $\mu$ is a Borel measure.

On the other hand, analogous representation results for {\em non-negative}  subharmonic functions of a given Markov processes do not seem to exist in  general form. Note that non-negativity is essential here: If $g$ is a subharmonic function,  although $-g(X)$ is a supermartingale, $-g$ is not excessive as it is  negative. Thus, the representation results for excessive functions are not applicable. One exception to this general rule occurs when the Markov process is transient and the subharmonic function $g$ is bounded by $K$. Then one can obtain a representation using the available theory for the excessive function $K-g$. However, this approach will fail when $g$ is unbounded or the Markov process is recurrent, which renders all excessive functions constant { (see also Remark \ref{r:naive} for  explicit and non-trivial examples illustrating this difficulty)}.   

Non-negative subharmonic functions are also known as `defective' functions (see p.31 of Dellacherie and Meyer \cite{DM-C}) and play an important role in Rost's solution to the Skorokhod embedding problem \cite{rost_stopping71}. Despite their abundance relative to excessive functions and their use in the potential theory, ``It is quite depressing to admit that one knows almost nothing about defective functions'' as Dellacherie and Meyer point out in \cite{DM-C} 

{ The main purpose of this study  is to fill a gap in this direction by establishing an integral representation for non-negative subharmonic functions of a regular one-dimensional diffusion $X$ on a given interval $\elr$. This is followed by an in-depth analysis of integral equations associated with strictly positive subharmonic functions. The main contributions of the paper will be summarised in the following paragraphs:
\subsection{Extremal subharmonic functions and last passage times}	Theorem \ref{t:choquet_ex}  identifies the {\em extremal} subharmonic functions, which ultimately allows a {\em Choquet-type} integral representation via Theorem \ref{t:choquet_gen}.  Evidently,  the  submartingales defined by the extremal subharmonic functions belong to Class $\Sigma$, (see  \cite{Nikeghbali06} and \cite{Nikeghbalimult}). This leads to  an alternative representation of minimal subharmonic functions in terms of last passage times. Given the close connections with submartingales of Class $\Sigma$ and Az\'ema supermartingales, { one obtains in particular that monotone subharmonic functions can be written as a mixture of Az\'ema {\em submartingales}   as long as the set of Az\'ema submartingales is not empty (see Remark \ref{r:Azema}).  Moreover, the representation in terms of last passage times indicates that this connection could still be useful in a multidimensional framework. The arguments can be also used to prove the analogous result that one-dimensional excessive functions  are in fact  a mixture of Az\'ema supermartingales, which seems to have been unnoticed in the literature.} 

\subsection{Integral equations for strictly positive subharmonic functions} To every strictly positive subharmonic function one can associate a {\em positive continuous additive functional} (PCAF) $A$ such that $g(X)\exp(-A)$ is a local martingale.  For historical reasons mentioned in Section \ref{s:choquet} such pairs $(g,A)$ are called \Ito-Watanabe pairs throughout the text.  This observation helps one to characterise all subharmonic functions appearing in \Ito-Watanabe pairs as solutions of integral equations by means of the Choquet representation from Theorem \ref{t:choquet_gen}.  Section \ref{s:IE} establishes that  any such subharmonic function can be written as a linear combination of monotone subharmonic functions that are solutions of particular integral equations. 

A family of integral equations for which solutions exist and can be used to generate {\em all} strictly positive subharmonic functions are studied in Section \ref{s:existence}. { The novelty of this section as opposed to the  `classical' approach of characterising the so-called fundamental solutions (e.g. as in \cite{dayanikRD} or \cite{BL}) as particular Laplace transforms in the spirit of \Ito {} and McKean \cite{IM} is that the  fundamental solutions herein are characterised as fixed points of monotone integral operators and therefore can be obtained after a straightforward and fast numerical algorithm.}

\subsection{Transient transformations of diffusions} Section \ref{s:transform} studies changes of measures (or {\em path transformations}) for diffusions via \Ito-Watanabe pairs. It is in particular shown that after these path transformations  the diffusion process ends up transient, thereby providing  a complete counterpart to {\em recurrent transformations} introduced in \cite{rectr} via $h(X)\exp(B)$, where $h$ is excessive and $B$ is a PCAF.  
\subsection{Connections to the optimal stopping problem with random discounting}
An important motivation to consider the path transformations developed in this paper is the optimal stopping problem with random discounting studied earlier by \cite{BL} and \cite{dayanikRD}. The problem of interest is to maximise $E^x[e^{-A_{\tau}}f(X_{\tau})]$ over all stopping times,
where $E^x$ is expectation with respect to $P^x$ that corresponds to the law of $X$ with $X_0=x$, $f$ is a reward function, and $A$ is a PCAF. Using the arguments of \cite{rectr} the above problem becomes equivalent to the solution of 
$$
\sup_{\tau}Q^x\left[\frac{f(X_{\tau})}{g(X_{\tau})}\right],
$$
where $Q^x$ is locally absolutely continuous with respect to $P^x$ and $g$ is a positive subharmonic function. Thus, the value function, after an explicit change of variable,  becomes the least concave majorant of $f/g$.   This idea is similar at heart to the approach first proposed by Beibel and Lerche \cite{BL} and later developed in further generality by Dayan\i k in \cite{dayanikRD} for one-dimensional diffusions (see also \cite{CPT} for another application of `removing the discounting' in the context of L\'evy processes).  { The main idea in all these works  is to simplify the problem by `removing' the discounting factor via a measure change. That is, find a subharmonic function $g$ so that $g(X)\exp(-A)$ is a local martingale. This allows for the reduction of the above optimal stopping problem to one without discounting, in which the value function can be identified as the concave envelope of a given function.} However, \cite{BL} and \cite{dayanikRD} only give an abstract definition in terms of the expectation of a multiplicative functional. On the other hand, the representation of strictly positive subharmonic functions established in Sections \ref{s:IE} and \ref{s:existence}  allow one to compute $g$ explicitly by solving an integral equation, { whose numerical solution is easy as observed above. 

The outline of the paper is as follows: Section \ref{s:prelim} introduces the set up and basic terminology that will be used in the paper. Section \ref{s:choquet} studies  the Choquet representation of non-negative subharmonic functions and its first consequences. Section \ref{s:IE} gives a complete characterisation  and uniqueness of solutions of the integral equations that are solved by semi-bounded subharmonic functions appearing in \Ito-Watanabe pairs. Section \ref{s:existence} establishes the existence of solutions for the integral equations of Section \ref{s:IE} and discusses its numerical solutions. The path transformations via \Ito-Watanabe pairs are studied in Section \ref{s:transform}, and Section \ref{s:conc} concludes.
\section{Preliminaries} \label{s:prelim}
Let $X=(\Om, \cF, \cF_t, X_t, \theta_t, P^x)$ be a regular  diffusion on $\bfE:=(\ell,r)$, where $ -\infty \leq \ell <r \leq \infty$, and $\mathcal{E}^u$ stands for the $\sigma$-algebra of universally measurable subsets of $\bfE$. The boundaries are assumed to be absorbing, i.e. if any of the boundaries are reached in finite time, the process stays there forever. We do not allow killing inside $\bfE$, which is in fact without loss of generality for the purposes of the present paper as explained in Remark \ref{r:killing}.  As usual, $P^x$ is the law of the process initiated at point $x$ at $t=0$ and $\zeta$ is its lifetime, i.e. $\zeta :=\inf\{t>0:X_{t} \in \{\ell,r\} \}$. The transition semigroup of $X$ will be given by the kernels $(P_t)_{t \geq 0}$ on $(\bfE, \cE^u)$ and $(\theta_t)_{t \geq 0}$ is the shift operator.  { The filtration $(\cF^*_t)_{t \geq 0}$ will denote the universal completion of the natural filtration of $X$ and $\cF^u$ is the $\sigma$-algebra generated by the maps $f(X_t)$ with $t\geq 0$ and $f$ universally measurable\footnote{The reader is referred to Chapter 1 of  \cite{GTMP}  for the details.}. We shall set $\cF_t:= \cF^*_{t+}$ for each $t\geq 0$ so that  $(\cF_t)_{t \geq 0}$ is right continuous}. The infinitesimal generator of $X$ will be denoted by $\cA$.

For $y \in (\ell,r)$ the stopping time  $T_y:=\inf\{t> 0: X_t=y\}$, where the infimum of an empty set  equals $\zeta$ by convention, is the first hitting time of $y$. Likewise $T_{ab}$ will denote the exit time from the interval $(a,b)$.  Clearly, one can extend the notion of `hitting time' to each of the boundary points, i.e. by allowing $y \in \{\ell,r\}$ in $T_y$, as the boundaries are absorbing. 

Such a one-dimensional diffusion is completely characterised by its  strictly increasing and continuous scale function $s$ and speed measure $m$.  The reader is referred to Chapter II of \cite{BorSal} for a concise treatment of these characteristics (see also \cite{EH} for a rigorous discussion of h-transformation of such diffusions, which is related to, yet different than, the type of conditioning that will be discussed in Section \ref{s:transform}). In particular, since the killing measure is null, the infinitesimal generator $\cA$ of the diffusion is given by $\cA=\frac{d}{dm}\frac{d}{ds}$. 

\begin{remark}
	It is worth emphasizing here that no assumption of absolute continuity with respect to the Lebesgue measure  is made for the scale function or the speed measure. That is, $X$ is not necessarily the solution of a stochastic differential equation. A notable example is {\em the skew Brownian motion} (see \cite{skewBM81}).
\end{remark}

The concept of a {\em positive continuous additive functional} will be playing a key role throughout the paper.
\begin{definition}
	\label{d:caf} A family $A=(A_t)_{t \geq 0}$ of functions from $\Omega$ to $[0,\infty]$ is called a {\em positive continuous additive functional} (PCAF) of $X$ if
	\begin{itemize}
		\item[i)] Almost surely the mapping $t \mapsto A_t$ is nondecreasing, (finite) continuous on $[0,\zeta)$, and $A_t = A_{\zeta-}$ for $t \geq \zeta$.
		\item[ii)] $A_t$ is $\cF_t$-measurable for each $t\geq 0$.
		\item[iii)] For each $t$ and $s$ $A_{t+s}=A_t + A_s \circ \theta_t$, a.s..
	\end{itemize}
\end{definition}
To each PCAF $A$ one can associate a {\em Revuz measure} $\mu_A$ defined on the Borel subsets of $\elr$  by
\be \label{d:Revmes}
\int_{\elr}f(y)\mu_A(dy)=\lim_{t \rar 0}t^{-1}E^m[\int_0^tf(X_s)dA_s],
\ee
where $f$ is a non-negative Borel function and $E^m[U]:=\int_{\elr}E^x[U]m(dx)$ for any non-negative random variable $U$. It must be noted that the Revuz measure depends on the choice of the speed measure. 

Moreover, in this one-dimensional setting $\mu_A$ will be a Radon measure\footnote{This  follows from Theorem 4.4 in \cite{BGafd64}. Note that the fine topology induced by $X$ coincides with the standard metric topology on $\elr$ (see, e.g., Exercise 10.22 in \cite{GTMP}). Moreover, one-dimensional diffusions are self-dual with respect to their speed measure. Thus, compact subsets of $\elr$ are (co-)special.}. 
\begin{remark}
	\label{r:killing} One possible use of continuous additive functionals is the construction of a diffusion with non-zero killing measure from a diffusion with the same scale and speed but no killing a described in Paragraph 22 of Chapter II in \cite{BorSal}. For this reason and given the nature of questions addressed in this paper, one indeed does not lose any generality by assuming a null killing measure.
\end{remark} 

As the killing measure is null, the potential density with respect to $m$ of a transient diffusion is given by
\[
u(x,y)= \lim_{a \rar \ell}\lim_{b \rar r} \frac{(s(x\wedge y)-s(a))(s(b)-s(x \vee y))}{s(b)-s(a)}, \qquad x, y \mbox{ in } \elr.
\]
In this case (see, e.g.,  Theorem VI.3.1 in \cite{BG}) for any non-negative Borel function $f$
\be \label{e:potentialA}
E^x\left[\int_0^{\zeta}f(X_t)dA_t\right]=\int_{\ell}^{r}u(x,y)f(y)\mu_A(dy).
\ee
Moreover, the finiteness of $A_{\zeta}$, or equivalently $A_{\infty}$, is completely determined in terms of $s$ and $\mu_A$. The following is a direct consequence of Lemma A1.7 in \cite{AS98} (see \cite{MU} for an analogous result and a different technique of proof in case of $dA_t =f(X_t)dt$ for some non-negative measurable $f$).
\begin{theorem}\label{t:Afinite} Let $A$ be a PCAF of $X$ with Revuz measure $\mu_A$. 
	\begin{enumerate}
		\item If $X$ is recurrent and $\mu_A(\bfE)>0$, then $A_{\zeta}=\infty$, a.s..
		\item If $s(\ell)>-\infty$, then on $[X_{\zeta-}=\ell]$, $A_{\zeta}=\infty$  a.s.  or $A_{\zeta}<\infty$  a.s. whether 
		\[
		\int_{\ell}^c(s(x)-s(\ell))\mu_A(dx)  
		\]
		is infinite or not for some $c \in \elr$.
		\item If $s(r)<\infty$, then on $[X_{\zeta-}=r]$ $A_{\zeta}=\infty$  a.s.  or $A_{\zeta}<\infty$  a.s.  whether 
		\[
		\int_c^r(s(r)-s(x))\mu_A(dx)
		\]
		is infinite or not for some $c \in \elr$.
	\end{enumerate}
\end{theorem}

If $A$ is a PCAF, it can be used to apply a time-change, the most prominent example of which appears when constructing a diffusion with given characteristics from a Brownian motion (see Chapter 5 in \cite{IM} for a detailed account). In particular, if one uses $A$ as `the clock,' behaviour of  a given diffusion at the boundary may change. The following  characterisation of the boundary behaviour will be useful in subsequent sections. 
\begin{definition} \label{d:CAFbc}
	Let $A$ be a PCAF with Revuz measure $\mu_A$ such that $\mu_A(\elr)>0$. Pick $b \in \elr$ and consider the integrals 
	\bea \label{e:Aboundaryx}
	\int_{\ell}^b \mu_A((z,b))s(dz) & \mbox{\phantom{bos}} &\left(\mbox{resp.} \int_b^r \mu_A((b,z))s(dz) \right)\\
	\label{e:Aboundarye}
	\int_{\ell}^b (s(b)-s(z))\mu_A(dz) & \mbox{\phantom{bos}} &\left(\mbox{resp.} \int_b^r ((s(z)-s(b))\mu_A(dz) \right).
	\eea
	\begin{description}
		\item[\bf $A$-regular:] $\ell$ (resp. $r$) is (an) $A$-regular (boundary) if both (\ref{e:Aboundaryx}) and (\ref{e:Aboundarye}) are finite.
		\item[\bf $A$-exit:]  $\ell$ (resp. $r$) is (an) $A$-exit (boundary) if (\ref{e:Aboundaryx}) is finite and (\ref{e:Aboundarye}) is infinite.
		\item[\bf $A$-entrance:]  $\ell$ (resp. $r$) is (an) $A$-entrance (boundary) if (\ref{e:Aboundaryx}) is infinite and (\ref{e:Aboundarye}) is finite.
		\item[\bf $A$-natural:]  $\ell$ (resp. $r$) is (an) $A$-natural (boundary) if both (\ref{e:Aboundaryx})  and (\ref{e:Aboundarye}) are infinite.
	\end{description}
\end{definition}
\begin{remark}
	\label{r:Aboundary} Note that (\ref{e:Aboundaryx}) is infinite if $s(\ell)$ (resp. $s(r)$) is so. In case of $s(\ell)$ being finite, the finiteness of (\ref{e:Aboundaryx}) is equivalent to that of $A_{\zeta}$ on $[X_{\zeta-}=\ell]$ (resp. $[X_{\zeta-}=r]$ via a straightforward integration by parts.
\end{remark}
\begin{remark} The boundary classification above is significantly different than the `usual' one (see, e.g., Paragraph 6 in Section II.1 in \cite{BorSal}). Consult Remark \ref{r:Anatural} to see this in an interesting example.
\end{remark}
\begin{proposition} \label{p:hittimes}
	Let $A$ be a PCAF with Revuz measure $\mu_A$ such that $\mu_A(\elr)>0$.\begin{enumerate}
		\item If $\ell$ (resp. $r$) is $A$-entrance, for any $y\in \elr$ there exists a $t_0>0$  such that
		\be \label{e:entrdef}
		\lim_{x \rar \ell} P^x(A_{T_y}<t_0)>0. \qquad \left(\mbox{resp. }  \lim_{x \rar r} P^x(A_{T_y}<t_0)>0.\right)
		\ee
		\item If $\ell$ (resp. $r$) is $A$-natural, for any $y\in \elr$ and $t>0$
		\be \label{e:natdef}
		\lim_{x \rar \ell} P^x(A_{T_y}<t)=0. \qquad \left(\mbox{resp. }  \lim_{x \rar r} P^x(A_{T_y}<t)=0.\right)
		\ee
	\end{enumerate}
\end{proposition}
\begin{proof}
	\begin{enumerate}
		\item Note that $s(\ell)$ is necessarily infinite and, therefore, for any $x<y$, $P^x(T_y<\zeta)=1$. Killing $X$ at $T_y$ yields a transient diffusion on $(\ell, y)$ with potential kernel
		\[
		v(x,z):= s(y)-s(x\vee z).
		\]
		In particular
		\[
		E^x[A_{T_y}]=\int_{\ell}^y (s(y)-s(x\vee z))\mu_A(dz).
		\]
		Observe by Chebyshev's inequality that 
		\[
		P^x(A_{T_y}>t)\leq \frac{E^x[A_{T_y}]}{t}.
		\]
		Also note that $\ell$ being $A$-entrance implies $\mu_A((\ell,y))<\infty$. Thus, by means of  the dominated convergence theorem one obtains
		\[
		\lim_{x \rar \ell}P^x(A_{T_y}>t)\leq \frac{1}{t}\int_{\ell}^y\lim_{x\rar \ell} (s(y)-s(x\vee z))\mu_A(dz)=\frac{1}{t}\int_{\ell}^y(s(y)-s(z))\mu_A(dz)<\infty.
		\]
		Hence, choosing $t$ big enough one shows $\lim_{x \rar \ell}P^x(A_{T_y}>t)<1$. 
		\item Again, killing $X$ at $T_y$ yields a transient diffusion with potential kernel $v$ that coincides with the one from the previous part if $s(\ell)=-\infty$. Otherwise, 
		\[
		v(x,z)=\frac{(s(x\wedge z)-s(\ell))(s(y)-s(x\vee z))}{s(y)-s(\ell)}.
		\]
		Define $h(x):=E^x[\exp(-A_{T_y})]$. Since $h$ is excessive, it must be a concave function of $s$. In particular, it is continuous and  $\lim_{x \rar \ell}h(x)$ exists. 
		
		It follows from Feynman-Kac's formula (see (8) on p. 119 of \cite{FP-Kac}) that $h$ solves
		\[
		h(x)=1-\int_{\ell}^y v(x,z)h(z)\mu_A(dz).
		\]
		Observe that under the assumption that $\ell$ is $A$-natural and $s(\ell)=-\infty$, $\int_{\ell}^{b} v(\ell,z)\mu_A(dz)=\infty$ for any $b\in (\ell,y)$. Thus, one necessarily obtains
		\[
		\lim_{x\rar \ell}h(x)=0.
		\]
		If $s(\ell)>-\infty$, then $\int_{\ell}^{b} v(x,z)\mu_A(dz)=\infty$ for any $x\in (\ell,b)$, leading to $	\lim_{x\rar \ell}h(x)=0.$
		Hence,
		\[
		\lim_{x \rar \ell} P^x(A_{T_y}<t)\leq e^t 	\lim_{x \rar \ell}E^x[\exp(-A_{T_y})]=0.
		\]
	\end{enumerate}
\end{proof}
\section{Non-negative subharmonic functions and Choquet representation} \label{s:choquet}
\begin{definition}
	A Borel function $g:\bfE \to \bbR$ is {\em subharmonic} if  for any  $x \in (\ell,r)$  $g(X^{T_{ab}})$ is a uniformly integrable $P^x$-submartingale whenever $\ell<a<b<r$. The class of non-negative subharmonic functions is denoted by $\cS$.  $\cS^+$  will be the set of elements of $\cS$ that are are strictly positive on $(\ell,r)$. 
\end{definition}
Since $g(x)\leq E^x[g(X_{T_{ab}})]=g(a)P^x(T_a<T_b)+g(b)P^x(T_b<T_a)=g(a)\frac{s(b)-s(x)}{s(b)-s(a)}+g(b)\frac{s(x)-s(a)}{s(b)-s(a)}$, one immediately deduces that $g$ must be a convex function of $s$ on the open interval $(\ell,r)$ to be subharmonic, which in particular entails that $g$ is an absolutely continuous function of $s$. { In particular, the right and left $s$-derivatives, $\frac{d^+g}{ds}$ and  $\frac{d^-g}{ds}$, exist (see Section 3 of the Appendix in \cite{RY} for a discussion of $s$-convexity).} 
\begin{remark}
	In the sequel whenever a convex function is considered on some open interval $(a,b)$ it will be automatically extended to $[a,b]$ by continuity.
\end{remark}

The main purpose of this section is to obtain Choquet-type integral representations for non-negative subharmonic functions. 
\begin{theorem}\label{t:choquet_ex}
	Define $\cK_0^+:=\{g \in \cS: g(\ell+)=0, \mbox{ $g$ is non-decreasing with } \frac{d^+g}{ds}\leq 1\}$ and $\cK_0^-:=\{g \in \cS: g(r-)=0, \mbox{ $g$ is non-increasing with } \frac{d^-g}{ds}\geq -1\}$. Then, the following hold:
	\begin{enumerate}
		\item The set of extremal elements of $\cK_0^+$ is given by 
		\bean
		\{ k (s(x)-s(y))^+: k\in (0,1] \mbox{ and } y \in (\ell, r)\},&\;& \mbox{if  } \ell =-\infty;\\
		\{ k (s(x)-s(y))^+: k\in (0,1] \mbox{ and } y \in [\ell, r)\},&\;& \mbox{if  } \ell >-\infty.
		\eean
		In particular, if $g \in \cS$ is non-constant, non-decreasing and its right $s$-derivative $\frac{d^+g}{ds}$ is bounded on $\bfE$, 
		\be
		g(x)=g(\ell) + \kappa (s(x)-s(\ell)) + \int_{\ell}^r (s(x)-s(y))^+ \mu(dy), \label{e:Choquet1}
		\ee
		where $\kappa =0$ if $s(\ell)=-\infty$ with the convention that $0 \times \infty=0$, and $\mu$ is a { finite} measure with $\mu(\{\ell,r\})=0$.
		\item The set of extremal elements of $\cK_0^-$ is given by 
		\bean
		\{ k (s(y)-s(x))^+: k\in (0,1] \mbox{ and } y \in (\ell, r)\}&\;& \mbox{if  } r =\infty;\\
		\{ k (s(y)-s(x))^+: k\in (0,1] \mbox{ and } y \in (\ell, r]\}&\;& \mbox{if  } r <\infty.
		\eean
		In particular, if $g\in \cS$ is non-constant, non-increasing and its left $s$-derivative $\frac{d^-g}{ds}$ is bounded on $\bfE$, 
		\be
		g(x)=g(r) + \kappa (s(r)-s(x)) + \int_{\ell}^r (s(y)-s(x))^+ \mu(dy), \label{e:Choquet2}
		\ee
		where $\kappa =0$ if $s(r)=\infty$ with the convention that $0 \times \infty=0$, and $\mu$ is a { finite} measure with $\mu(\{\ell,r\})=0$.
	\end{enumerate}
\end{theorem}
\begin{proof} Only the proof of the first statement will be given as the proof is the same modulo obvious modifications for the second statement.
	
	Without loss of generality suppose that $s(x)=x$ for all $x \in \bfE$  and observe that the set $\cK_0^+$ is a compact convex metrisable subset of the space of continuous functions on $\bfE$ with locally uniform topology by a version of Arzela-Ascoli theorem. Thus, by Choquet's theorem (see, e.g., Theorem on p.14 of Phelps \cite{PhelpsChoquet}) any element of $\cK^+_0$ is representable by its extremal elements. 
	
	Next, let $g \in \cK_0^+$ be an extremal element\footnote{By convention, $g$ cannot be identically $0$.} and denote $\frac{dg^+}{dx}$ by $g^+$.  { Since, for any $z \in [\ell,r)$ and $y \in (\ell,r)$,
		\[
		g^+(z)=g^+(z)\chf_{z<y} + g^+(y)\chf_{z\geq y} + \chf_{z\geq y}(g^+(z)-g^+(y)),
		\] 
		one obtains the following decomposition from the fact that $g(x)=\int_{\ell}^x g^+(z)dz$:}
	\bean 
	g&=&g_1(\cdot;y) +g_2(\cdot;y), \quad \mbox{ where}\\
	g_1(x;y)&:=&\int_{\ell}^x \left\{g^+(z)\chf_{z<y} + g^+(y)\chf_{z\geq y}\right\}dz, \mbox{ and}\nn\\
	g_2(x;y)&:=&\int_{\ell}^x\chf_{z\geq y}(g^+(z)-g^+(y))dz.\nn
	\eean
	Observe that, for each $y \in (\ell,r)$, the functions $g_1(\cdot;y)$ and $g_2(\cdot;y)$ belong to $\cK_0^+$ { since $g^+\leq 1$ and $g^+(z)$ is nondecreasing in $z$ by convexity}. 
	
	Since $g$ is extremal, one necessarily has
	\[
	g_1(\cdot;y)=k_1(y) g \mbox{ and } g_2(\cdot;y)=k_2(y) g
	\]
	such that $k_1(y)+k_2(y)=1$ for every $y \in \bfE$.  Moreover, that $g_2(\cdot;y)=0$ on $(\ell, y]$ implies either $g=0$ on $(\ell, y]$ or $k_2(y)=0$. On the other hand, there exists a $y^* \in [\ell,r)$ such that $g(y^*)=0$ and $g(x)>0$ for $x>y^*$ since $g$ is not identically $0$. Thus,  for each $y>y^*$, $k_2(y)=0$, which in turn implies 
	\be \label{e:gextrem}
	g(x)= g_1(x;y)= g(y)+ (x-y) g^+(y), \quad x\geq y >y^*.
	\ee
	In particular, $g^+$ is constant on $(y^*,r)$. However, since $g^+$ is right continuous, one in fact obtains  $g^+(y)=g^+(y^*)$ for all $y>y^*$. Note that this in particular implies that in case of $\ell=-\infty$,  $y^*>\ell$. Indeed, since $g$ is non-negative and convex,  $g^+(-\infty)=0$ due to the fact that $g$ is non-decreasing\footnote{ $g^+(-\infty):=\lim_{y\rar -\infty}g^+(y)$ exists by the monotonicity of the right derivative.}.  Therefore, (\ref{e:gextrem}) can be rewritten as
	\[
	g(x)=g(y^*) +k(x-y^*)=k(x-y^*), \qquad x\geq y^*,
	\]
	for some $k\in (0,1]$ since $g^+\leq 1$ by hypothesis. Since $0\leq g(x)\leq g(y^*)=0$ for $x<y^*$ by the monotonicity, this proves the desired representation for the extremal elements of $\cK^+_0$ as $y^*$ can be any element of $[\ell,r)$. Note that $(s -s(r))^+\equiv 0$ is the trivial  extremal element. 
	
	Aforementioned Choquet's theorem now associates a probability measure $\hat{\mu}$  on the Borel subsets of the extremal points in $\cK_0^+$ with right derivatives converging to $1$ at $r$, denoted with $\cK_{e}^+$, to any  $g\in \cK_0^+$ with $g^+(r-)=1$ such that
	\[
	g= \int_{\cS_{e}^+} \xi\hat{\mu}(d\xi).
	\]
	Moreover, $y \mapsto (\cdot-y)^+$ is a one-to-one continuous mapping from $\bfE$ (resp. $[\ell, r)$ if $\ell >-\infty$)  onto $\cK_{e}^+$. Thus, after a change of variable, we obtain another probability measure $\tilde{\mu}$ on the Borel subsets of $\bfE$ (resp. $[\ell, r)$ if $\ell >-\infty$) such that
	\[
	g(x)= \int_{\ell}^r (x-y)^+ \tilde{\mu}(dy).
	\]
	In particular, if $\ell>-\infty$, the above can be rewritten as
	\[
	g(x)= \tilde{\mu}({\ell}) (x-\ell) + \int_{\ell+}^r (x-y)^+ \tilde{\mu}(dy),
	\]
	which in turn yields the representation (\ref{e:Choquet1}) by considering $\frac{g - g(\ell)}{g^+(r-)}$.
	
	{ To show that $\mu$ is finite, observe (\ref{e:Choquet1})  that
		\[
		g^+(x)= \kappa +\int_{\ell+}^{x}\mu(dy)=\kappa +\mu((\ell,x]).
		\]
		Thus, the claim follows from the boundedness of $g^+$.}
\end{proof}

To extend the above representation to all monotone functions in $\cS$, one needs a uniqueness result.
\begin{proposition} \label{p:Choquniq}
	The constant $\kappa$ and the measure $\mu$ appearing in the representations (\ref{e:Choquet1}) and (\ref{e:Choquet2}) are uniquely defined by $g$. In particular, for any $\ell <a <x<b<r$,
	\[
	E^x[g(X_{T_{ab}})]=g(x)+ \int_{a}^b \frac{(s(x\wedge y)-s(a))(s(b)-s(x\vee y))}{s(b)-s(a)}\mu(dy),
	\]
	and, therefore, $g(X)-B$ is a $P^x$-local martingale for any $x \in \bfE$, where
	\[
	B:=\int_{\ell}^r\mu(dy)L^y,
	\]
	and $L^y$ is the diffusion local time at $y$.
\end{proposition}
\begin{proof}
	See the Appendix.
\end{proof}

\begin{theorem}\label{t:choquet_gen}
	\begin{enumerate}
		\item If $g \in \cS$ is non-constant and non-decreasing, 
		\be
		g(x)=g(\ell) + \kappa (s(x)-s(\ell)) + \int_{\ell}^r (s(x)-s(y))^+ \mu(dy), \label{e:Choquetgen1}
		\ee
		where $\kappa =0$ if $s(\ell)=-\infty$ with the convention that $0 \times \infty=0$, and $\mu$ is a { Radon} measure with $\mu(\{\ell,r\})=0$.
		\item If $g \in \cS$ is non-constant and non-increasing, 
		\[
		g(x)=g(r) + \kappa (s(r)-s(x)) + \int_{\ell}^r (s(y)-s(x))^+ \mu(dy), 
		\]
		where $\kappa =0$ if $s(r)=\infty$ with the convention that $0 \times \infty=0$, and $\mu$ is a { Radon} measure with $\mu(\{\ell,r\})=0$.
		\item Consequently, $g \in \cS$ if and only if there exist { Radon} measures $(\mu_i)_{i=1}^2$ on the Borel subsets of $(\ell,r)$ and non-negative constants $\kappa_1$ and $\kappa_2$ such that
		\begin{align} 
			g(x)&=\alpha + \kappa_1 (s(x)-s(\ell)) + \kappa_2 (s(r)-s(x)) \nn \\
			&+ \int_{\ell}^r (s(x)-s(y))^+ \mu_1(dy) +\int_{\ell}^r (s(y)-s(x))^+ \mu_2(dy),  \label{e:Choquetgen}
		\end{align}
		where $\alpha\geq 0$ and $\kappa_1=0$ (resp. $\kappa_2=0$) if $s(\ell)=-\infty$ (resp. $s(r)=\infty$).
		
		Moreover, given $g\in \cS$ with decomposition  (\ref{e:Choquetgen}), $g-B$ is a $P^x$-local martingale for any $x\in \elr$, where $B$ is a PCAF with Revuz measure $\mu_1+\mu_2$.
		
		In particular, if $g\in \cS$ is not monotone and $c^*$ is such that $g(c^*)=\inf_{x \in \bfE}g(x$), the measures $\mu_1$ and $\mu_2$ and the constants $\kappa_1$ and $\kappa_2$ above can be chosen  such that $\mu_1((\ell,c^*))=\mu_2((c^*,r))=\kappa_1=\kappa_2=0$, in which case  $\alpha=g(c^*)$.
	\end{enumerate}
\end{theorem}
\begin{proof}
	Only the proof of the first and the last statement will be given. To this end observe that the right $s$-derivative of $g$ is bounded on $(\ell, b)$ for any $b<r$. Thus, stopping $X$ at $T_b$, one obtains a diffusion living on $(\ell,b]$ (possibly including $\ell$ in case it is an absorbing boundary) for which $g$ is a subharmonic function.  Thus, Theorem \ref{t:choquet_ex} yields coefficients $\kappa_b$ and measures $\mu_b$ indexed by $b \in (\ell,r)$ such that
	\[
	g(x)=g(\ell) + \kappa_b (s(x)-s(\ell)) + \int_{\ell}^r (s(x)-s(y))^+ \mu_b(dy) \mbox{ on } (\ell,b].
	\]
	Next, the uniqueness of the above representations via Proposition \ref{p:Choquniq} yields $\kappa_b =\kappa_{b'}$ and $\mu_b=\mu_{b'}$ on the Borel subsets of $(\ell,b]$ whenever $\ell<b<b'<r$. { In particular, $\mu_b((\ell,b])<\infty$  since $\mu_{b'}$ is a finite measure by Proposition \ref{p:Choquniq}.} Thus, there exists a measure $\mu$ on the Borel subsets of $(\ell,r)$ such that $\mu$ agrees with $\mu_b$ on the Borel subsets of $(\ell,b]$ for any $b\in (\ell,r)$, { which in turn implies $\mu$ is a Radon measure}. Therefore, there exits a $\kappa$ ($=0$ if $\ell=-\infty$) such that
	(\ref{e:Choquetgen1}) holds for all $x \in (\ell, r)$, which can be extended to hold for $x=r$ by continuity in case $r<\infty$.
	
	Finally, first observe that if $g$ is of the form (\ref{e:Choquetgen}), it is $s$-convex and non-negative, therefore $g\in \cS$. Conversely, if $g\in \cS$ is monotone, the claimed representation is a consequence of the first two parts.
	
	Next, suppose that $g$ is not monotone and let $g_1^+:=\max\{0,\frac{d^+g}{ds}\}$, $g_2^+:=\min\{0,\frac{d^+g}{ds}\}$. Pick $c^* \in (\ell,r)$ where $g$ attains its minimum and observe that $g_1^+=0$ on $(l,c^*)$ and $g_2^+=0$ on $(c^*,r)$. Now define $g_i$ for $i=1,2$ by 
	$g_1(x)=g(c^*)+\int_{c^*}^x g_1^+(y)s(dy)$ and $g_2(x)=\int^{x}_{c^*} g_2^+(y)s(dy)$. Note that  $g_1$ is non-decreasing and $g_2$ is non-increasing such that $g=g_1 +g_2$.  Moreover, the previous two parts yield the constants $\kappa_i$ and the Radon measures $\mu_i$ such that
	\begin{align*}
		g_1(x)&=g(c^*)+ \kappa_1(s(x)-s(\ell)) +\int_{\ell}^r (s(x)-s(y))^+ \mu_1(dy),\\
		g_2(x)&=\kappa_2(s(r)-s(x)) +\int_{\ell}^r (s(y)-s(x))^+ \mu_2(dy),
	\end{align*}
	
	Direct computations show that in above representations $\frac{d^+g_1(\ell+)}{ds}=\kappa_1$ and $\frac{d^-g_2(r-)}{ds}=\kappa_1$. On the other hand, by construction $\frac{d^+g_1(\ell+)}{ds}=\frac{d^-g_2(r-)}{ds}=0$. Thus,
	\[
	g_1(x)=g(c^*)+\int_{\ell}^r (s(x)-s(y))^+ \mu_1(dy) , \;\forall x\in \elr.
	\]
	Note that the above in conjunction with the fact that $g_1$ is constant on $(\ell,c^*)$ implies that $\mu_1((\ell,c^*))=0$. Similar argument shows that $\mu_2((c^*,r))=0$, which in turn yields
	\[
	g(x)=g(c^*)+\int_{\ell}^r (s(x)-s(y))^+ \mu_1(dy)+\int_{\ell}^r (s(y)-s(x))^+ \mu_2(dy),
	\]
	where $\mu_1((\ell,c^*))=\mu_2((c^*,r))=0$. 
	
	Moreover, since $g_1$ and $g_2$ are monotone, there exists CAFs $B_1$ and $B_2$ with Revuz measures $\mu_1$ and $\mu_2$, respectively, such that $g_i-B_i$ is a $P^x$-local martingale for $i=1,2$ in view of Proposition \ref{p:Choquniq}. Therefore, $B=B_1+B_2$ is a PCAF with Revuz measure $\mu_1+\mu_2$ and $g-B$ is a $P^x$-local martingale.
\end{proof}

\subsection{Connection with last passage times}
Note that $(s(X)-s(y))^+$ is a local submartingale and by means of Tanaka-Meyer formula it is easy to see that $(s(X)-s(y))^+= M + V$, where $M$ is a local martingale and $V$ is continuous and increasing such that $dV_t$ is carried by the set $\{t: s(X_t)-s(y)=0\}$. Thus, after a straightforward translation $(s(X)-s(y))^+$ can be viewed as a {\em submartingale of class $\Sigma$} using the terminology of Nikeghbali \cite{Nikeghbali06}. As observed therein, and also in a related context earlier by Az\'ema, Meyer and Yor  \cite{AMYrelmar}, this paves the way to a representation in terms of last passage times.
\begin{lemma} \label{l:LPsub}
	Consider $(y,z)\in [\ell,r]^2$  such that $y \neq z$ and define 
	\[
	\Lambda^{y,z}:=\sup\{0<t<T_z: X_t=y\},
	\]
	with the convention\footnote{For notational simplicity the superscript $z$ will be dropped if $z\in \{\ell,r\}$.} that the supremum of the empty set is $0$.  Then the following statements are valid:
	\begin{enumerate}
		\item Let $x \in \elr$ and suppose $z> x\vee y$ with $s(z)<\infty$. Then,
		\be \label{e:LPRup}
		(s(x)-s(y))^+ = (s(z)-s(y))P^x(T_z<T_{\ell}, \Lambda^{y,z}=0).
		\ee	
		Alternatively,
		\be \label{e:LPRupa}
		(s(X_t)-s(y))^+=E^x\left[(s(X_{T_z})-s(y))^+ \chf_{[\Lambda^{y,z}\leq t]}\big|\cF_t\right] \mbox{ on } [t<T_z].
		\ee
		\item Let $x \in \elr$ and suppose $z< x\wedge y$ with $s(z)>-\infty$. Then,
		\be \label{e:LPRdown}
		(s(y)-s(x))^+ = (s(y)-s(z))P^x(T_z<T_{r}, \Lambda^{y,z}=0).
		\ee	
		Alternatively,
		\be \label{e:LPRdowna}
		(s(y)-s(X_t))^+=E^x\left[(s(y)-s(X_{T_z}))^+ \chf_{[\Lambda^{y,z}\leq t]}\big|\cF_t\right] \mbox{ on } [t<T_z].
		\ee
	\end{enumerate}
\end{lemma}
\begin{proof}
	Only the first statement will be proven as the other can be proven similarly. 
	
	{ 	Observe that $M=(s(z)-s(X_{t\wedge T_z}))_{t\geq 0}$ is non-negative continuous local martingale.  Thus, applying Theorem 2.5 in \cite{PRYoption} to $M$ and $K=s(z)-s(y)$ yields \eqref{e:LPRup}.  In particular, 
		\[
		\begin{split}
			(s(x)-s(y))^+&=E^x\left[(s(X_{T_z})-s(y))^+ \chf_{[\Lambda^{y,z}=0]}\right] =E^x\left[(s(z)-s(y)) \chf_{[T_z<T_{\ell},\Lambda^{y,z}=0]}\right] \\
			&=(s(z)-s(y))P^{x}(T_z<T_{\ell}, \Lambda^{y,z}=0).
		\end{split}
		\]
}
\end{proof}
A direct but interesting corollary of the above result is that when $X$ is transient and $X_{\infty}$ takes values in a singleton, certain monotone subharmonic functions can be written as a mixture of {\em Az\'ema submartingales}.
\begin{theorem} Suppose $X$ is transient and define for any $y\in \elr$ the Az\'ema submartingale
	\[
	Z^y_t=P^x(\Lambda^y\leq t|\cF_t).
	\]
	Then the following statements are valid:
	\begin{enumerate}
		\item Suppose  $s(\ell)=-\infty$. If $g \in \cS$ is non-constant and non-decreasing, 
		\be
		g(X_t)=g(\ell) + \int_{\ell}^r Z^y_t  \mu(dy), \label{e:Azema1}
		\ee
		where $\mu$ is a Radon measure with $\mu(\{\ell,r\})=0$.
		\item  Suppose  $s(r)=\infty$. If $g \in \cS$ is non-constant and non-increasing, 
		\be
		g(X_t)=g(r) +\int_{\ell}^r Z^y_t \mu(dy), \label{e:Azema2}
		\ee
		where $\mu$ is a Radon measure with $\mu(\{\ell,r\})=0$. 
	\end{enumerate}
\end{theorem}
{ \begin{proof} Only the first statement will be proven as the statements are analogous. 
		
		The representation \eqref{e:Choquetgen1} yields a Radon measure $\mu_0$ that doesn't  charge $\{\ell,r\}$ such that
		\[
		g(x)=g(\ell)+\int_{\ell}^r (s(x)-s(y))^+ \mu_0(dy).
		\]
		
		Next observe that $s(r)<\infty$ by the assumed transience and consider \eqref{e:LPRupa} with $z=r$, which implies in view of $P^x(X_{\infty}=r)=1$ that 
		\[
		(s(X_t)-s(y))^+=(s(r)-s(y))P^x(\Lambda^y\leq t|\cF_t).
		\]
		Defining $\mu(dy):=(s(r)-s(y))\mu_0(dy)$ yields the claim.
	\end{proof}
	{ \begin{remark} \label{r:Azema}
			Note that a transience assumption is necessary to obtain any representation result involving Az\'ema submartingales. Indeed, if $X$ is recurrent, $\Lambda^y=\infty$, $P^x$-a.s. for each $x \in \elr$. However, this implies $Z^y\equiv 0$ for all $y\in \elr$.
	\end{remark}}
	An analogous result can be obtained for  excessive functions using the Az\'ema {\em supermartingale} $P^x(\Lambda^y>t|\cF_t)=1-Z^y_t$.
	\begin{theorem} Suppose $X$ is transient and  $h$ is excessive. 	Then the following statements are valid:
		\begin{enumerate}
			\item If $s(\ell)=-\infty$,
			\be
			h(X_t)=\kappa (s(r)-s(x)) + \int_{\ell}^r (1-Z^y_t)  \mu(dy), \label{e:Azema1h}
			\ee
			where $\mu$ is a finite measure with $\mu(\{\ell,r\})=0$ and $\kappa\geq 0$.
			\item  If $s(r)=\infty$, 
			\be
			h(X_t)=\kappa (s(x)-s(\ell)) +\int_{\ell}^r (1-Z^y_t) \mu(dy), \label{e:Azema2h}
			\ee
			where $\mu$ is a finite measure with $\mu(\{\ell,r\})=0$ and $\kappa\geq 0$. 
		\end{enumerate}
	\end{theorem}
	\begin{proof}
		As before, only the first statement will be proven, in which case $s(r)<\infty$ by the assumed transience.
		
		It follows from Paragraph 30 in Section II.5 of \cite{BorSal} that there exists a finite measure $\mu_0$ not charging $\{\ell,r\}$ and a non-negative constant $\kappa$ that
		\[
		h(x)=\kappa(s(r)-s(x)) +\int_{\ell}^r u(x,y)\mu_0(dy).
		\]
		Fix $y\in \elr$ and observe that $s(x\vee y)=(s(x)-s(y))^+ + s(y)$.  Thus, using the arguments in the proof of Lemma \ref{l:LPsub} with $z=r$, one obtains
		\[
		s(X_t\vee y) =s(y) + (s(r)-s(y)) P^x(\Lambda^y\leq t|\cF_t)=s(y)+(s(r)-s(y)) Z_t^y.
		\]
		Since $u(X_t,y)=s(r)-s(X_t\vee y)$, the claim follows.
	\end{proof}
	\begin{remark} \label{r:naive}
		The last two theorems illustrate the difficulty with the naive idea of obtaining the integral representation for subharmonic functions from that of superharmonic ones. Note that the representing measure in \eqref{e:Azema1h} is finite while $\mu$ appearing in \eqref{e:Azema1} is only locally finite.
\end{remark}}
\subsection{\Ito-Watanabe pairs}
By definition, given any subharmonic function $g$, $g(X)$ is a local submartingale and, therefore, there exists a unique PCAF $B$ with $B_0=0$ such that $g(X) -B$ is a $P^x$-local martingale for any $x \in (\ell,r)$ by a Markovian version of the Doob-Meyer decomposition (see Theorem 51.7 in \cite{GTMP}).   If $g$ is further supposed to be in $\cS^+$, then a simple integration by parts argument yields a unique $A$ such that $g(X)\exp(-A)$ is a local martingale, where $A$ is an adapted,  continuous and increasing process with $A_0=0$. Clearly, this $A$ is defined by its initial condition and $g(X_t)dA_t=dB_t$. 

The above argument gives a multiplicative decomposition for $g\in \cS^+$ as a product of a local martingale and an increasing process. The strict positivity is essential and the following summarises the above discussion.
\begin{theorem} \label{t:submvanish}
	Let $g \in \cS$ and suppose further that $g$ is not identically $0$. There exists a PCAF $A$  such that $g(X)\exp(-A)$ is a $P^x$-local martingale for every $ x\in (\ell,r)$ if and only if $g$ never vanishes on $(\ell,r)$.
\end{theorem}
\begin{proof}
	What remains to be proven is the implication that the existence of an $A$ with above properties implies strict positivity of $g$. To this end suppose the closed set $Z:=\{x\in (\ell,r):g(x)=0\}$ is not empty. By the regularity of $X$ $P^x(T_Z<\zeta)>0$ for any $x \in (\ell,r)$. Since $g(X)\exp(-A)$ is a supermartingale being a non-negative local martingale, it will remain zero on $[T_Z, \zeta)$, which in turn implies $X$ does not leave the set $Z$ on $[T_Z, \zeta)$ since $\exp(-A_t)> 0$ on $[t<\zeta]$. One then deduces via the strong Markov property of $X$ that  $P^z(T_y <\zeta)=0$ for any $z \in Z$ and $y \in Z^c$. However, this contradicts the regularity of $X$.
\end{proof}


Consequently, the local submartingale $g(X)$ has a multiplicative decomposition as a product of a local martingale and an increasing process. Such multiplicative decompositions in the context of Markov process goes back to the work of \Ito { and}  Watanabe \cite{IWmult} who studied  multiplicative decompositions of supermartingales and their use in the study of subprocesses. This historical note motivates the following definition. 
\begin{definition}
	$(g,A)$ is called an {\em \Ito-Watanabe pair} if $A$ is a PCAF, $g\in \cS^+$,and $g(X)\exp(-A)$ is a $P^x$-local martingale for all $x \in \elr$.
\end{definition}

\begin{definition}
	A function $g$ is said to be {\em uniformly integrable near $\ell$ (resp. $r$)} if the family $\{g(X^{T_b}_{\tau}):\tau \mbox{ is a stopping time}\}$ is $P^x$-uniformly integrable for any $x <b$ (resp. $x>b$).  $g$ is said to be {\em semi-uniformly integrable} if it is uniformly integrable near $\ell$ or $r$. 
\end{definition}
\begin{proposition} \label{p:bcg1}
	Suppose that $(g,A)$ is an \Ito-Watanabe pair and $g$ is bounded near $\ell$ (resp. $r$). Then the following statements are valid:
	\begin{enumerate}
		\item $g(\ell)=0$ (resp. $g(r)=0$) if $\ell$ (resp. $r$) is $A$-natural or $A$-exit.
		\item $g(\ell)>0$ (resp. $g(r)>0$) if $\ell$ (resp. $r$) is $A$-entrance.
	\end{enumerate}
\end{proposition}
\begin{proof}
	Only the  statements pertaining to $\ell$ will be proven as the proofs are analogous in the other case.
	\begin{enumerate}
		\item Suppose that $\ell$ is $A$-natural. Then by the assumed uniform integrability, 
		\[
		g(x)\leq C E^x \left[\exp(-A_{T_y})\right],\; x\leq y,
		\]
		for any $y\in \elr$ and $t>0$, where $C=C\max\{g(\ell), g(y)\}$. Moreover, the expectation on the right hand side of the above converges to $0$ as $x\rar \ell$ as observed in the proof of (\ref{e:natdef}).
		
		If $\ell$ is $A$-exit, then necessarily $s(\ell)$ is finite. Similar to the above
		\be \label{e:glimitxt}
		\begin{split}
			g(x)&= g(\ell) E^x\left[\exp(-A_{\zeta}) \chf_{[T_y=\zeta]}\right]+  g(y) E^x\left[\exp(-A_{T_y}) \chf_{[T_y<\zeta]}\right], \qquad x<y\in \elr;\\
			&\leq  g(\ell) E^x\left[\exp(-A_{\zeta}) \chf_{[T_y=\zeta]}\right]+  g(y)\frac{(s(x)-s(\ell))}{s(y)-s(\ell)}.
		\end{split}
		\ee
		
		Let $f(x):= E^x\left[\exp(-A_{\zeta}) \chf_{[T_y=\zeta]})\right]$ and consider the $h$-transform by $h(x)=\frac{s(y)-s(x)}{s(y)-s(\ell)}$ that defines the probability measure $(Q^x)_{x \in (\ell,y)}$ via
		\[
		h(x) Q^x(B):=E^x(h(X_{t\wedge T_y})\chf_B), \quad B\in \cF_t.
		\]
		Then the coordinate process $X$ is a regular diffusion under $Q^x$ that lives on $(\ell,y)$ and converges to $\ell$ at its lifetime. In particular
		\[
		f(x)=h(x)Q^x(\exp(-A_{\zeta})), \quad x \in (\ell,y).
		\]
		
		Using the explicit formulas for the potential kernel and the speed measure after an $h$ transform (see Paragraph 31 on p. 33 of \cite{BorSal}), one obtains by the Feynman-Kac formula (see (8) on p. 119 of \cite{FP-Kac}) that 
		\[
		f(x)=h(x)\left(1-\int_{\ell}^y \frac{(s(x\wedge z)-s(\ell))(s(y)-s(x\vee z))}{s(y)-s(\ell)}\frac{s(y)-s(z)}{s(y)-s(x)}f(z)\mu_A(dz)\right),
		\]
		since the Revuz measure of $A$ under $Q^x$ is given by $h^2(z)\mu_A(dz)$. Since $f$ is $P$-excessive, $f(\ell+)$ exists. Moreover, the fact that $\ell$ is $A$-exit implies $f(\ell+)$ must be $0$ in view of the above representation. Therefore, taking limits as $x\rar \ell$ in (\ref{e:glimitxt}), one establishes $g(\ell)=0$.
		\item Since $s(\ell)=-\infty$, for any $y \in \elr$, $g(x)=g(y)E^x\left[\exp(-A_{T_y})\right]$ whenever $x\in (\ell,y)$. By Proposition \ref{p:hittimes} there exists $t_0>0$ such that $\lim_{x\rar \ell}P^x(A_{T_y}<t_0)>0$. Thus,
		\[
		\lim_{x\rar \ell} g(x)\geq g(y)\lim_{x\rar \ell}E^x\left[\exp(-A_{T_y})\chf_{[A_{T_y}<t_0]}\right]\geq g(y)e^{-t_0}\lim_{x\rar \ell}P^x(A_{T_y}<t_0)>0.
		\]
	\end{enumerate}
\end{proof}
\begin{remark} \label{r:Anatural}
	The above result might appear wrong at  first sight. Indeed, $g(x):=e^x+1$ is an increasing subharmonic function for Brownian motion on $\bbR$ and, yet, it does not vanish at the natural boundary $-\infty$. However, the key point here is that $-\infty$ is not $A$-natural for $A$ being the PCAF that makes $g(X)\exp(-A)$ a local martingale. A quick calculation shows that up to a scaling factor $\mu_A(dx)=(1+e^{-x})^{-1}dx$ and thus $\int_{-\infty}^b (b-x)\mu_A(dx)<\infty$. In fact, $-\infty$ is $A$-entrance, which is consistent with $g(-\infty)>0$. 
\end{remark}
Observe that if $(g,A)$ is an \Ito-Watanabe pair and $g$ is non-decreasing, Proposition \ref{p:Choquniq}, Theorem \ref{t:choquet_gen}, and the uniqueness of the Doob-Meyer decomposition imply that
\[
g(x)=g(\ell) +\kappa (s(x)-s(\ell))+ \int_{\ell}^r (s(x)-s(y))^+g(y)\mu_A(dy),
\]
with the usual disclaimer that $\kappa=0$ when $s(\ell)=-\infty$. Next, considering an arbitrary $c \in \bfE$, one obtains after straightforward calculations the following integral equation:
\be \label{e:ginteq}
g(x)=g(c)+\kappa (s(x)-s(c)) +\int_{\ell}^r \left(s(x\vee y)-s(c\vee y)\right)g(y)\mu_A(dy).
\ee
Formally speaking, given a PCAF $A$ with Revuz measure $\mu_A$, one should expect that solution of  (\ref{e:ginteq}) yields a $g \in \cS^+$  such that $(g,A)$ is an  \Ito-Watanabe pair. Analogously, if $g$ in Theorem \ref{t:choquet_gen} is non-increasing, then
\[
g(x)=g(c)+\kappa (s(c)-s(x)) +\int_{\ell}^r \left(s(c\wedge y)-s(x\wedge y)\right)g(y)\mu_A(dy).
\] 

The focus of the next section will be the analysis of integral equations whose solutions lead to subharmonic functions appearing in \Ito-Watanabe pairs given a PCAF of $X$. Note that in general there is no uniqueness for such $g$. For instance, if $A_t=t$, then the increasing and decreasing solutions of $\cA g =g$ belong to $\cS^+$ and $g(X)\exp(-A)$ are local martingales.   
\begin{remark}
	One can also consider the following killing procedure to construct the subharmonic function in a given \Ito-Watanabe pair $(g,A)$: Let $\mathbf{e}$ be a unit exponential random variable independent of $X$ and define a new diffusion $Y$ by $Y_t=X_t$ if $\mathbf{e}>A_t$ and sending $Y$ to the cemetery state on the event $[\mathbf{e}\leq A_t]$. The resulting process is a regular diffusion and it is characterised uniquely upon identifying its scale function, speed measure and killing measure. Thus, finding $g$ amounts to identifying the harmonic functions of $Y$, which typically follows from the general representation results for Markov processes via the theory of Marting boundary (see, e.g., Theorem 14.8 in \cite{ChungWalsh}). However, this requires the identification of the potential kernel of $Y$  and its limit at the Martin boundary. Since the killing measure of $Y$  is not null, the simple  representation of the potential kernel in terms of the scale function is no longer valid (see Paragraph 11 of \cite{BorSal}), and there doesn't seem to be a formula for the minimal harmonic functions for diffusions with general characteristics.
	
	Thus, one is left with finding the fundamental solutions of an ODE associated to the infinitesimal generator of $Y$ using the approach described in Paragraph 10 of \cite{BorSal} (for $\alpha=0$). The difficulty with this approach is that it requires specifications of  boundary conditions depending on the boundary behaviour of $Y$, which is in general different then that of $X$ for $A$ being a general PCAF.  The approach of this paper, on the other hand,  is direct and does not require any verification of boundary conditions.
\end{remark}
\section{Integral equations for positive subharmonic functions}\label{s:IE}
The following key lemma, whose proof is delegated to the Appendix, will be useful in proving the representation results for \Ito-Watanabe pairs.
\begin{lemma}\label{l:mainR}
	Let $\mu_A$ be a Radon measure on $\elr$ and  $c\in \elr$. Consider a function $g$ that is a solution\footnote{Any solution is implicitly assumed to be integrable in the sense that $\int_{\ell}^r |v_c(x,y)g(y)|\mu_A(dy)<\infty$ for all $x \in \elr$.} of 
	\be \label{e:genelIE}
	g(x)=g(c)+ \kappa (s(x)-s(c))+ \int_{\ell}^r v_c(x,y)g(y)\mu_A(dy),
	\ee
	where either
	\[
	v_c(x,y)= s(x\vee y)-s(c\vee y), \;  \forall (x,y)\in (\ell,r)^2\;
	\]
	or
	\[
	v_c(x,y)=s(c\wedge y)-s(x\wedge y),\; \forall (x,y)\in (\ell,r)^2.
	\]
	Define $O^+:=\{x\in \elr:g(x)>0\}$  and $O^-:=\{x\in \elr:g(x)<0\}$. Then, following statements are valid:
	\begin{enumerate}
		\item For any $\ell <a<x<b<r$
		\be 
		E^x[g(X_{T_{ab}})]=g(x)+ \int_a^b \frac{(s(x\wedge y)-s(a))(s(b)-s(x\vee y))}{s(b)-s(a)}g(y)\mu_A(dy)\label{e:EhabR}
		\ee
		\item $g$ is $s$-convex on $O^+$ and $s$-concave on $O^-$.
		\item If  $v_c(x,y)= s(x\vee y)-s(c\vee y)\, \left(\mbox{resp. }s(c\wedge y)-s(x\wedge y)\right)$,  then $g(\ell+)$ (resp. $g(r-)$) is finite if $s(\ell)$ (resp. $s(r)$) is.  Moreover, $\kappa=0$ and $g(\ell+)$ is finite if $g$ is uniformly integrable near $\ell$ (resp. $r$) and $s(\ell)=-\infty$ (resp. $s(r)=\infty$).
		\item  If $\kappa= 0$, $g$ does not change sign in $\elr$.
		\item $g$ is differentiable with respect to $s$ from  left and  right with following derivatives:
		\be \label{e:gsder}
		\begin{split}
			\frac{d^+g(x)}{ds}&=\left\{\ba{ll}\kappa +\int_{\ell}^x g(y)\mu_A(dy), & \mbox{if }v_c(x,y)= s(x\vee y)-s(c\vee y); \\
			\kappa - \int_{x+}^r g(y)\mu_A(dy), & \mbox{if }v_c(x,y)= s(c\wedge y)-s(x\wedge y).
			\ea\right.
			\\
			\frac{d^-g(x)}{ds}&=\left\{\ba{ll}\kappa +\int_{\ell}^{x-} g(y)\mu_A(dy), & \mbox{if }v_c(x,y)= s(x\vee y)-s(c\vee y); \\
			\kappa - \int_{x}^r g(y)\mu_A(dy), & \mbox{if }v_c(x,y)= s(c\wedge y)-s(x\wedge y).
			\ea  \right.
		\end{split}
		\ee
		Consequently, $g$ is differentiable with respect to $s$ at $x$ if $\mu_A(\{x\})=0$ or $g(x)=0$. Moreover, $\frac{d^+g(\ell)}{ds}$ (resp. $\frac{d^-g(r)}{ds}$) exists and satisfies the above formula whenever $g(\ell)<\infty$ (resp. $g(r)<\infty$).
	\end{enumerate} 
\end{lemma}

\begin{remark} \label{r:gsign}
	If $\kappa\neq 0$, $g$ can change sign. Indeed, suppose $\elr=(-1,1), \, \mu_A(dy)=dy$, and $s(x)=x$. Then, $g(x)=\sinh(x)$ solves
	\[
	g(x)=\cosh(-1) \sinh(x) +\int_{\ell}^r(x\vee y-y^+)g(y)\mu_A(dy).
	\]
	Clearly, this is linked to a Brownian motion on $(-1,1)$. $\sinh(B_t)\exp(-\frac{t}{2})$ is a local martingale that hits $0$ infinitely many times. 
\end{remark}

\begin{theorem}\label{t:main1}
	Let $g$ be a  Borel measurable function on $\bfE$ and $A$ a  PCAF with Revuz measure $\mu_A$ with $\mu_A(\bfE)>0$.  Consider the sets
	\begin{align*}
		G_{\ell}&:=\left\{g\in \cS^+: g \mbox{ is u.i. near }\ell, \, \big|\frac{d^+g(\ell+)}{ds}\big|<\infty, \mbox{ and } g(r-)>\inf_{x\in \elr}g(x).\right\}\\
		G_{r}&:=\left\{g\in \cS^+: g \mbox{ is u.i. near }r, \, \big|\frac{d^-g(r-)}{ds}\big|<\infty, \mbox{ and } g(\ell+)>\inf_{x\in \elr}g(x).\right\}.
	\end{align*}
	Then the following are equivalent:
	\begin{enumerate}[label=(\arabic*),ref=(\arabic*)]
		\item \label{t:eq1} $(g,A)$ is an \Ito-Watanabe pair and $g \in G_{\ell}$ (resp.  $g \in G_{r}$)
		\item \label{t:eq2} $g$ is non-negative, not identically $0$ and  solves the integral equation
		\be \label{e:IEh}
		\begin{split}
			g(x) &=g(c)+ \kappa_1 (s(x)-s(c))+ \int_{\ell}^r\left(s(x\vee y)-s(c\vee y)\right)g(y)\mu_A(dy), \mbox{ if } g \in G_{\ell};\\
			g(x)&=g(c)+ \kappa_2 (s(x)-s(c))+ \int_{\ell}^r\left(s(c\wedge y)-s(x\wedge y)\right)g(y)\mu_A(dy), \mbox{ if } g \in G_{r},
		\end{split}
		\ee
		where $\kappa_1=0$ (resp. $\kappa_2=0$) if $s(\ell)=-\infty$ (resp. $s(r)=\infty$).
	\end{enumerate}
\end{theorem}

\begin{proof}
	See the Appendix.
\end{proof}

A special case of the above result is obtained when $c$ is replaced by $\ell$ or $r$.  Indeed, one deduces the following corollary by means of monotone convergence, { which identifies the measure appearing in the  Choquet representation from (\ref{e:Choquetgen}) with $g\cdot \mu_A$.}

\begin{corollary} \label{c:main1}
	Let $g$ be a  Borel measurable function on $\bfE$ and $A$ a  PCAF with Revuz measure $\mu_A$ with $\mu_A(\bfE)>0$.  Let  the sets $G_{\ell}$ and $G_r$ be as in Theorem \ref{t:main1}.  
	Then the following are equivalent:
	\begin{enumerate}
		\item  $(g,A)$ is an \Ito-Watanabe pair and $g \in G_{\ell}$ (resp.  $g \in G_{r}$)
		\item  $g$ is non-negative, not identically $0$ and  solves the integral equation
		\be \label{e:IEh_choquet}
		\begin{split}
			g(x) &=g(\ell)+ \kappa_1 (s(x)-s(\ell))+ \int_{\ell}^r\left(s(x)-s(y)\right)^+ g(y)\mu_A(dy), \mbox{ if } g \in G_{\ell};\\
			g(x)&=g(r)- \kappa_2 (s(r)-s(x))+ \int_{\ell}^r\left(s( y)-s(x)\right)^+g(y)\mu_A(dy), \mbox{ if } g \in G_{r},
		\end{split}
		\ee
		where $\kappa_1=0$ (resp. $\kappa_2=0$) if $s(\ell)=-\infty$ (resp. $s(r)=\infty$).
	\end{enumerate}
\end{corollary}
\begin{remark} \label{r:GlGr}
	Note that any solution of the first integral equation in (\ref{e:IEh_choquet}) as well as in (\ref{e:IEh}) always belongs to $G_{\ell}$.  Analogously, solutions of the second integral equation is in $G_r$.
\end{remark}
It was observed in Lemma \ref{l:mainR} that the coefficient $\kappa_1$  (resp. $\kappa_2$) in  (\ref{e:IEh_choquet}) must vanish when $s(\ell)$  (resp. $s(r)$) is infinite.  Thanks to (\ref{e:IEh_choquet}) more can be said about these coefficients as well as $g(\ell)$ and $g(r)$.
\begin{proposition}
	\label{p:kappas}
	Consider a non-negative $g$ that is not identically $0$ and solves (\ref{e:genelIE}) for some $c \in [\ell,r]$. Then the following statements are valid.
	\begin{enumerate}
		\item Suppose $v_c(x,y)= s(x\vee y)-s(c\vee y)$ for all $(x,y)\in \elr^2$ and $g$ is u.i. near $\ell$. If  $\ell$ is $A$-exit or $A$-natural,		then $g(\ell)=0$. Similarly, $\kappa=0$ if $\ell$ is $A$-entrance or $A$-natural.
		\item Suppose $v_c(x,y)= s(c\wedge y)-s(x\wedge y)$ for all $(x,y)\in \elr^2$ and $g$ is u.i. near $r$. If  $r$ is $A$-exit or $A$-natural, then $g(r)=0$. Similarly, $\kappa=0$ if $r$ is $A$-entrance or $A$-natural.
	\end{enumerate}
\end{proposition}
\begin{proof}
	Only the proof of the first statement will be given as before.
	
	Observe that if $g$ satisfies (\ref{e:genelIE}) for some $c\in \elr$, $g(\ell)<\infty$ by Lemma \ref{l:mainR}. Thus, it satisfies (\ref{e:IEh_choquet}), too. First suppose $\int_{\ell}^{b}(s(b)-s(y))\mu_A(dy)=\infty$ for some $b\in \elr$; that is, $\ell$ is $A$-exit or $A$-natural. Then by Corollary \ref{c:main1} and  Proposition \ref{p:bcg1}, $g(\ell)=0$. 
	
	Note that if $s(\ell)=-\infty$, $\int_{\ell}^b \mu_A((z,b))s(dz)=\infty$. In this case $\kappa=0$ follows from Lemma \ref{l:mainR}. If  $s(\ell)$ is finite,  $\int_{\ell}^b \mu_A((z,b))s(dz)=\infty$ implies
	\[
	\int_{\ell}^{x}(s(y)-s(\ell))\mu_A(dy) =\infty 
	\]
	for all $x$.  Suppose that $\kappa$ is not $0$ and note that $\kappa$ must be positive if $g(\ell)=0$ to ensure that $g$ is non-negative since $g$ is $s$-convex by Lemma \ref{l:mainR} and Corollary \ref{c:main1},  and   $\frac{d^+g}{ds}(\ell)=\kappa$ by another application of Lemma \ref{l:mainR}. 
	
	Thus, one can find $\delta>0$ and $k>0$ such that  $g(x)> k (s(x)-s(\ell))$ on  $(\ell, \ell+\delta)$. This in turn implies
	\[
	\int_{\ell}^{\ell+\delta}\left(s(\ell+\delta)-s(y)\right) g(y)\mu_A(dy)\geq k\int_{\ell}^{\ell+\delta}\left(s(\ell+\delta)-s(y)\right)(s(y)-s(\ell)) \mu_A(dy) =\infty,
	\]
	which is a contradiction. Hence, $\kappa=0$.
\end{proof}

The integral equation (\ref{e:genelIE}), in particular (\ref{e:IEh_choquet}), typically needs two independent boundary or initial conditions to admit a unique solution. Fixing the value of $g(c)$ in (\ref{e:genelIE}) handles one of these conditions. However, (\ref{e:gsder}) also shows that $\frac{d^+g(\ell)}{ds}=\kappa$ when $g$ is uniformly integrable near $\ell$. That is, there is a second initial boundary condition implicit in the equation and one may expect uniqueness by fixing the value of $g(c)$. This is not always the case as the following example shows.
\begin{example}
	Consider the integral equation
	\[
	g(x)= \int_0^{\infty}(x-y)^+ g(y)\frac{2}{y^2}dy.
	\]
	Clearly, $g(x)=k x^2$ is a non-negative solution of the above for any $k \geq 0$. Note that $0$ is $A$-natural, where $A$ is the PCAF for Brownian motion on $(0,\infty)$ with Revuz measure $\mu(dy)=\frac{2}{y^2}$ dy.
\end{example}
Despite the above example a uniqueness result may be obtained for a large class of integral equations considered.
\begin{theorem}
	\label{t:uniqueIE}
	Consider the following integral equation	
	\be \label{e:ieunique}
	g(x)=a+ \kappa (s(x)-s(c)) +\int_{\ell}^r v_c(x,y)g(y)\mu_A(dy),
	\ee
	where $\mu_A$ is the Revuz measure associated with a PCAF $A$ with $\mu_A(\bfE)>0$,  $v_c(x,y)$ is either $s(x\vee y)-s(c\vee y)$ for all $(x,y) \in \elr \times \elr$ or $s(c\wedge y)-s(x \wedge y)$ for all $(x,y)\in \elr \times \elr$.  
	
	\begin{enumerate}
		\item Suppose that $c\in \elr, a\geq 0$ and $\kappa\in \bbR$. Then there exists at most one non-negative solution to (\ref{e:ieunique}) such that $g$ is uniformly integrable near $\ell$ (resp. $r$) if $v_c(x,y)=s(x\vee y)-s(c\vee y)$ (resp. $v_c(x,y)=s(c\wedge y)-s(x \wedge y))$.
		\item Suppose that $c=\ell$ (resp. $c=r$) if  $v_c(x,y)=s(x\vee y)-s(c\vee y)$ (resp. $v_c(x,y)=s(c\wedge y)-s(x \wedge y))$. Assume further that $\ell$ (resp. $r$) is {\em not} an $A$-natural boundary. Then there exists at most one non-negative solution to (\ref{e:ieunique}) such that $g$ is uniformly integrable near $\ell$ (resp. $r$) if $v_c(x,y)=s(x\vee y)-s(c\vee y)$ (resp. $v_c(x,y)=s(c\wedge y)-s(x \wedge y))$ for any $a\geq 0$ and $\kappa \geq$ (resp. $\kappa \leq 0$).
	\end{enumerate}
\end{theorem}
\begin{proof}
	Proof will be given when $v_c(x,y)=s(x\vee y)-s(c\vee y)$, the other case being analogous. First consider the case $s(\ell)=-\infty$ and $c>\ell$. Recall that $c <r$ if $s(r)=\infty$ and it, otherwise, is allowed to take the value $r$.  Let $f$ and $g$ be two solutions of (\ref{e:ieunique}) that are uniformly integrable near $\ell$. Then, for any $x \in (\ell,c)$,
	\[
	f(x)-g(x)=E^x\left[(f(X_{T_c})-g(X_{T_c}))\exp(-A_{T_c})\right]=0,
	\]
	since $P^x(T_c<T_{\ell})=1$ when $s(\ell)=-\infty$ under the hypothesis on $c$. This shows that $f$ and $g$ coincide for any $x <c$.  This completes the proof of this case if $c=r$
	
	Next suppose $c<r$ and  consider $a<x<c<y<r$. Using the semi-uniform integrability of $f-g$, one can then conclude
	\[
	0=f(x)-g(x)=E^x\left[(f(X_{T_{ay}})-g(X_{T_{ay}}))\exp(-A_{T_{ay}})\right]=E^x\left[\chf_{[T_y<T_a]}(f(y)-g(y))\exp(-A_{T_{ay}})\right].
	\]
	Hence, $f$ and $g$ coincide on $(c,r)$, too.
	
	Now, suppose $s(\ell)>-\infty$, while still assuming $c>\ell$, and $f$ and $g$ are two solutions of (\ref{e:ieunique}). Then, (\ref{e:gsder}) yields \[
	\kappa= \frac{d^+g(\ell)}{ds}= \frac{d^+f(\ell)}{ds}.
	\]
	Define $h=f-g$ and observe that $h$ satisfies
	\[
	h(x)= \int_{\ell}^rv_c(x,y) h(y)\mu_A(dy), \qquad h(c)=0.
	\]
	Then, Lemma \ref{l:mainR} shows that $h$ does not change its sign on $(\ell,r)$. Without loss of generality suppose $h\geq 0$ on $(\ell,r)$.  Another application of Lemma \ref{l:mainR} now yields $h$ is $s$-convex. Moreover,
	$\frac{d^+h(\ell)}{ds}=\frac{d^+f(\ell)}{ds}-\frac{d^+g(\ell)}{ds}=0$. However, together with the condition that $h(c)=0$, this implies $h$ must be identically $0$ on $[\ell,c]$. That is, $f$ and $g$ coincide on $[\ell,c]$. The same martingale argument above  shows that they coincide on $[\ell,r)$.
	
	Now consider the remaining case $c=\ell$ and define as above $h=f-g$, where $f$ and $g$ are two solutions. Then, $h$ solves
	\[
	h(x)=\int_{\ell}^x (s(x)-s(y))h(y)\mu_A(dy). 
	\]
	Again, we may assume $h\geq 0$ by invoking Lemma \ref{l:mainR}.  
	
	Since $\ell$ is not $A$-natural, at least one of the following integrals is finite for all $b\in \elr$:
	\[
	\int_{\ell}^b (s(b)-s(y))\mu_A(dy) \qquad \int_{\ell}^b (s(y)-s(\ell))\mu_A(dy).
	\]
	Suppose the former is finite. One has
	\[
	0\leq h(x)\leq \int_{\ell}^x h(y)(s(b)-s(y))\mu_A(dy).
	\]
	A general version of Gronwall's inequality (see Exercise V.15 in \cite{Pro}) yields $h\equiv 0$ on $[\ell, b)$, which completes the proof in this case since $b$ was arbitrary.
	
	Similarly,  in case the latter is finite, first note that $\frac{d^+h}{ds}(\ell)$ is well-defined and $0$, which in turn yields that $\frac{h}{s-s(\ell)}$ is bounded on $[\ell,b]$ for any $b\in \elr$. Therefore,
	\[
	0\leq \frac{h(x)}{s(x)-s(\ell)}\leq \int_{\ell}^x h(y)\frac{s(x)-s(y)}{s(x)-s(\ell)})\mu_A(dy)\leq \int_{\ell}^x \frac{h(y)}{s(y)-s(\ell)}(s(y)-s(\ell))\mu_A(dy).
	\]
	Another application of Gronwall inequality yields the claim.
\end{proof}

Due to Theorem \ref{t:submvanish} $(g,A)$ is an \Ito-Watanabe pair only if $g$ never vanishes on $\elr$. Thus the only cases of interest where uniqueness may be an issue occur when one considers the integral equations of (\ref{e:IEh_choquet}) with $g(\ell)=\kappa_1=0$ (resp. $g(r)=\kappa_2=0$) and the boundary is $A$-natural. The following corollary to the above theorem shows that the uniqueness is achieved as soon as the value of the function at an intermediary point is fixed.

\begin{corollary}
	\label{c:unique}
	Let $\mu_A$ be a Radon measure on $\elr$   with $\mu_A(\elr)>0$.  Let $c \in \elr$  and $a>0$. Then there exist at most one solution to the following integral equations
	\[
	\begin{split}
		g(x) &=\int_{\ell}^r\left(s(x)-s(y)\right)^+ g(y)\mu_A(dy), \mbox{ such that } g(c)=a;\\
		g(x)&=\int_{\ell}^r\left(s( y)-s(x)\right)^+g(y)\mu_A(dy), \mbox{ such that } g(c)=a
	\end{split}
	\]
\end{corollary}
\begin{proof}
	Consider the first equation and observe that
	\[
	g(x)=a +\int_{\ell}^r\left(s(x\vee c)-s(y\vee c)\right) g(y)\mu_A(dy).
	\]
	Thus, the uniqueness follow from Theorem \ref{t:uniqueIE}.  The second equation is treated similarly since in that case
	\[
	g(x)=g(c)+\int_{\ell}^r\left(s(y\wedge c)-s(x\wedge c)\right) g(y)\mu_A(dy).
	\]
\end{proof}

The following integration-by-parts type result regarding the solutions of (\ref{e:ieunique}) will be instrumental in Section \ref{s:transform}. {  Recall that $X$ is not assumed to be a semimartingale. Thus, one needs to argue via the representation of the infinitesimal generator in terms of the scale function and the speed measure.}
\begin{theorem} \label{t:ibp}
	Suppose  $g$ solves (\ref{e:ieunique}). Then
	\[
	dg(X_t)s(X_t)=s(X_t)dg(X_t)+g(X_t)ds(X_t)+ \frac{d^-g(X_t)}{ds}dB_t,
	\]
	where $B$ is a PCAF whose Revuz measure is $2s(dy)$.
\end{theorem}
\begin{proof}
	First note that there exists a PCAF $B^0$ such that $s^2(X)-B^0$ is a local martingale by Theorem 51.2 in \cite{GTMP}. Thus, if $\ell <a<b<r$ then
	\[
	E^x[s^2(X_{T_{ab}})]=s^2(x)+E^x[B^0_{T_{ab}}]=s^2(x)+\int_a^b u_{ab}(x,y)\mu^0(dy),
	\]
	where $\mu^0$ is the Revuz measure of $B^0$ and $u_{ab}(x,y)= \frac{(s(x\wedge y)-s(a))(s(b)-s(x \vee y))}{s(b)-s(a)}$. Then repeating the same calculations in the proof of Theorem VII.3.12 in \cite{RY}, after replacing $Af(y)m(dy)$ therein by $\mu^0(dy)$, one obtains
	\[
	\frac{ds^2(x)}{ds}-\frac{ds^2(y)}{ds}=\int_x^y \mu^0(dy).
	\]
	That is, $\mu^0(dy)= 2s(dy)$. 
	
	Moreover, for any $y \in (l,r)$, \Ito-Tanaka formula (see, e.g., Theorem 68 in Chap. IV of \cite{Pro}) in conjunction with $d[s(X),s(X)]_t=dB^0_t$ yields
	\[
	ds(X_t) s(X_t \vee y)= s(X_t \vee y)ds(X_t)+ s(X_t) ds(X_t \vee y) +  \chf_{[X_t> y]}dB^0_t.
	\]
	Thus, if $v_c(x,y)=s(x\vee y)-s(c\vee y)$,
	\[
	dg(X_t)s(X_t)= s(X_t)dg(X_t)+ g(X_t)ds(X_t)+ \kappa dB^0_t+\int_{\ell}^{X_t-} g(y)\mu_A(dy) dB_t^0.
	\]
	However, $\int_{\ell}^{x-} g(y)\mu_A(dy)=\frac{d^-g(x)}{ds}$ by (\ref{e:gsder}), which	establishes the claim. The case of $v_c(x,y)=s(c \wedge y)-s(c\wedge x)$ is treated similarly. 
\end{proof}
\section{Existence of solutions and further properties} \label{s:existence}
Lemma \ref{l:mainR} establishes that any semi-uniformly integrable $g$ appearing in an \Ito-Watanabe pair is semi-bounded. Combined with the $s$-convexity property this entails in particular that $g$ is also monotone on $(\ell,c)$ and $(c,r)$ for some $c\in [\ell,r]$.  This section will construct semi-bounded and monotone solutions of (\ref{e:genelIE}) for all $c\in [\ell,r]$ whenever they exist.

\begin{theorem} \label{t:existenceNoNat} Let $\mu_A$ be a Radon measure on $\elr$ such that $\mu_A(\elr)>0$ and consider the integral equations
	\bea
	g(x) &=&a+ \kappa (s(x)-s(\ell))+ \int_{\ell}^r\left(s(x)-s(y)\right)^+ g(y)\mu_A(dy),  \label{e:choquet1}\\
	g(x)&=& a+ \kappa (s(r)-s(x))+ \int_{\ell}^r\left(s( y)-s(x)\right)^+g(y)\mu_A(dy). \label{e:choquet2}
	\eea
	\begin{enumerate}
		\item If for some $b\in \elr$
		\be \label{e:hypoex}
		\int_{\ell}^b(s(b)-s(y))\mu_A(dy)<\infty\; \left(\mbox{ resp. } \int_b^r(s(y)-s(b))\mu_A(dy)<\infty \right), 
		\ee
		there exists a non-negative nondecreasing (resp. nonincreasing) solution to (\ref{e:choquet1}) (resp. \ref{e:choquet2}) for $a>0$ and $\kappa= 0$. Moreover, under this condition a solution with the same properties exists for $a>0, \kappa>0$ if one further assumes $s(\ell)$ (resp. $s(r)$) is finite.
		\item If $s(\ell)$ (resp. $s(r)$) is finite and for some $b\in \elr$
		\[
		\int_{\ell}^b (s(y)-s(\ell))\mu_A(dy)<\infty \;  \left(\mbox{ resp. } \int_b^r (s(r)-s(y))\mu_A(dy)<\infty \right),
		\]
		there exists a non-negative nondecreasing (resp. nonincreasing) solution to (\ref{e:choquet1}) (resp. \ref{e:choquet2}) for $a=0$ and $\kappa >0$.
	\end{enumerate}
\end{theorem}
\begin{proof}
	\begin{enumerate}[leftmargin=*]
		\item 	Consider (\ref{e:choquet1}) and given $a>0$ and $\kappa\geq 0$ define the operator $T$ acting on non-negative measurable functions on $[\ell,r]$ by
		\[
		Tg(x):=a+ \kappa (s(x)-s(\ell))+ \int_{\ell}^r\left(s(x)-s(y)\right)^+ g(y)\mu_A(dy).
		\]
		Observe that $Tg$ is a non-negative continuous function and $T$ is a monotone operator in the sense that $Tg\geq Tf$ on $\elr$ if $g\geq f$ is.  Set $g_0\equiv a$ and $g_n=Tg_{n-1}$ for $n\geq 1$. Since $g_1\geq a$, one deduces by induction that $g_n$ is increasing in $n$. Also note that (\ref{e:hypoex}) implies $\int_{\ell}^b (s(b)-s(x))g_n(x)\mu_A(dx)<\infty$ for all $b\in \elr$ and $n\geq 0$ since $g_n$s are continuous with $g_n(\ell)=a$. Moreover,
		\[
		g_{n}(x)\leq g_{n+1}(x)\leq a+ \kappa (s(x)-s(\ell))+ \int_{\ell}^r\left(s(x)-s(y)\right)^+ g_n(y)\mu_A(dy)
		\]
		implies that
		\[
		g_n(x)\leq \left(a+ \kappa (s(x)-s(\ell))\right) \exp\left(\int_{\ell}^b\left(s(b)-s(y)\right)\mu_A(dy) \right), \forall x\leq b,
		\]
		for any $b\in \elr$ by Gronwall's inequality. This shows that $g_n$s are uniformly bounded on $[\ell,b]$ for any $b\in \elr$. Thus, $g(x):=\lim_{n\rar \infty}g_n(x)$ exists and is finite for any $x\in [\ell, r)$. Moreover, it satisfies (\ref{e:choquet1}). That $g$ is monotone follows from Lemma \ref{l:mainR} and the previously observed relationship between (\ref{e:choquet1}) and (\ref{e:genelIE}) with $c \in \elr$. Finally note that $\kappa$ can be taken to be non-zero when $s(\ell)$ is finite in above.
		
		Construction of a solution to (\ref{e:choquet2}) is done similarly.
		
		\item The proof follows similar lines to that of the first part. Set $g_0=\kappa (s(x)-s(\ell))$, define $g_n$ analogously using (\ref{e:choquet1}) after setting $a=0$ and note that $\int_{\ell}^b g_1(x)\mu_A(dx)$ is finite since  $\int_{\ell}^b (s(x)-s(\ell))\mu_A(dx)$ is. Then it follows by induction that $\int_{\ell}^b g_n(x)\mu_A(dx)<\infty$. Moreover,
		\[
		\frac{g_n(x)}{s(x)-s(\ell)}\leq \frac{g_{n+1}(x)}{s(x)-s(\ell)}\leq \kappa + \int_{\ell}^x \frac{g_n(y)}{s(y)-s(\ell)}(s(y)-s(\ell))\mu_A(dy).
		\]
		It again follows from Gronwall's inequality that $\frac{g_n(x)}{s(x)-s(\ell)}$ is uniformly bounded near $\ell$. Thus, since $g_n$s are increasing $g=\lim_{n \rar \infty}g_n$ exists, is finite on $(\ell,r)$ that satisfies (\ref{e:choquet1}) with the stated properties by Lemma \ref{l:mainR}.
	\end{enumerate}
\end{proof}
\begin{remark}
{	Note that the proof of the above theorem gives the iterative algorithm to construct the solutions. This algorithm converges since  the corresponding equation has a unique solution as established by Theorem \ref{t:uniqueIE}. Moreover, the algorithm is easy to implement  as the underlying operator $T$ is monotone.}
\end{remark}

\begin{lemma}[Comparison lemma] Let $g_1$ and $g_2$ be two solutions of (\ref{e:choquet1}), where the coefficients $(a,\kappa)$ are replaced by $(a_1, \kappa_1)$ and $(a_2,\kappa_2)$, respectively. Suppose $a_1\leq a_2$ and $\kappa_1\leq \kappa_2$ such that at least one of the inequalities is strict. Then $g_1<g_2$ on $\elr$.
	
	An exact analogue of this comparison also holds for the solutions of (\ref{e:choquet2}).
\end{lemma}
\begin{proof}
	Let $g:=g_2-g_1$ and observe that $g$ satisfies
	\[
	g(x)=a_2-a_1 + (\kappa_2-\kappa_1)(s(x)-s(\ell)) +\int_{\ell}^r\left(s(x)-s(y)\right)^+ g(y)\mu_A(dy).
	\]
	If $a_2-a_1>0$,  Lemma \ref{l:mainR} yields $g$ is strictly positive and convex on a right neighbourhood of $\ell$. Since $\kappa_2-\kappa_1\geq 0$, this entails $g$ is strictly positive and convex by the same lemma. 
	
	If $a_2=a_1$ and $\kappa_2-\kappa_1>0$, then Lemma \ref{l:mainR} shows that $g$ is strictly increasing on a right neighbourhood of $\ell$. Consequently, it is convex and increasing on the same neighbourhood and, hence, on the whole $\elr$.
\end{proof}

In view of the above comparison result, Theorem \ref{t:existenceNoNat} and Theorem \ref{t:uniqueIE}  the following corollary is now immediate.
\begin{corollary}
	\label{c:existenceNoNat} Let $\mu_A$ be a Radon measure on $\elr$ such that $\mu_A(\elr)>0$. Then the following statements are valid:
	\begin{enumerate}
		\item Suppose that $\ell$ (resp. $r$) is $A$-regular. Then there exists a unique solution to (\ref{e:choquet1}) (resp. \ref{e:choquet2}) for any $a\geq 0$ and $\kappa \geq 0$. The solution is non-decreasing (resp. non-increasing) and never vanishes on $\elr$ if at least one of $a$ and $\kappa$ is non-zero.
		\item Suppose that $\ell$ (resp. $r$) is $A$-entrance. Then there exists a unique solution to (\ref{e:choquet1}) (resp. \ref{e:choquet2}) for any $a\geq 0$ and $\kappa= 0$. The solution is non-decreasing (resp. non-increasing) and never vanishes on $\elr$ if at least one of $a$ and $\kappa$ is non-zero.
		\item Suppose that $\ell$ (resp. $r$) is $A$-exit. Then there exists a unique solution to (\ref{e:choquet1}) (resp. \ref{e:choquet2}) for any $a=0$ and $\kappa\geq 0$. The solution is non-decreasing (resp. non-increasing) and never vanishes on $\elr$ if at least one of $a$ and $\kappa$ is non-zero.
	\end{enumerate}
\end{corollary}
\begin{remark}
	Note that in each case considered in Corollary \ref{c:existenceNoNat} the unique solution is given by $ g\equiv 0$ if $a=\kappa=0$. When one of the coefficients is strictly positive, the unique solution can be found by employing the iterative construction of Theorem \ref{t:existenceNoNat} combined with a diagonal argument in view of the Comparison Lemma.
\end{remark}
If $\ell$ is $A$-natural, it was observed in Proposition \ref{p:kappas} that $a$ and $\kappa$ must both vanish in (\ref{e:choquet1}), in which case the construction of a solution  via the algorithm employed in Theorem \ref{t:existenceNoNat} is no longer applicable. However, the following familiar probabilistic construction yields the desired solution.
\begin{theorem} \label{t:existenceNat}
	Suppose that  $A$ is a PCAF with Revuz measure $\mu_A$ such that $\mu_A(\bfE)>0$. Then, for any $\alpha \in (0,\infty)$ and $c \in \elr$ the following hold:
	\begin{enumerate}
		\item \label{i:semiinf:case1} Suppose that $\ell$ is $A$-natural. The increasing function
		\[
		g(x):=\left\{\ba{ll}
		\alpha E^x[\chf_{[T_c<T_{\ell}]}\exp(-A_{T_c})],& x\leq c,\\
		\frac{\alpha}{E^c[\chf_{[T_x<T_{\ell}]}\exp(-A_{T_x})]}, & x>c,
		\ea\right.
		\]
		is the unique solution of (\ref{e:choquet1}) with $a=\kappa=0$ such that $g(c)=\alpha$.
		\item  Suppose that $r$ is $A$-natural. The decreasing function
		\[
		g(x):=\left\{\ba{ll}
		\alpha E^x[\chf_{[T_c<T_{r}]}\exp(-A_{T_c})],& c< x,\\
		\frac{\alpha}{E^c[\chf_{[T_x<T_{r}]}\exp(-A_{T_x})]}, & x\leq c,
		\ea
		\right.
		\]
		is the unique solution of  (\ref{e:choquet2}) with $a=\kappa=0$ such that $g(c)=\alpha$.
	\end{enumerate}
\end{theorem}
\begin{proof} Without loss of generality assume $\alpha=1$.
	\begin{enumerate}[leftmargin=*]
		\item  Let $y >x \vee c$. Suppose $c <x$. Then, 
		\[
		E^c[\chf_{[T_y<T_{\ell}]}\exp(-A_{T_y})]=E^c[\chf_{[T_x<T_{\ell}]}\exp(-A_{T_x})]E^x[\chf_{[T_y<T_{\ell}]}\exp(-A_{T_y})]
		\]
		by the strong Markov property. Via similar considerations when $x \leq c$, one thus arrives at
		\[
		g(x)=\frac{E^x[\chf_{[T_y<T_{\ell}]}\exp(-A_{T_y})]}{E^c[\chf_{[T_y<T_{\ell}]}\exp(-A_{T_y})]}.
		\]
		In particular, $g_r(X)\exp(-A)$ is a bounded martingale when stopped at $T_y$. Since $T_y$ increases to $\zeta$ as $y \rar r$, this shows $(g_r,A)$ is an \Ito-Watanabe pair bounded at $\ell$. Thus, it follows from Proposition \ref{p:bcg1}, Theorem \ref{t:main1}, and Corollary \ref{c:unique} that $g$ is the unique solution of the stated equation.
		
		\item Repeat the above  starting with  $y<x\wedge c$ and observing
		\[
		g_{\ell}(x)=\frac{E^x[\chf_{[T_y<T_{r}]}\exp(-A_{T_y})]}{E^c[\chf_{[T_y<T_{r}]}\exp(-A_{T_y})]}.
		\]
	\end{enumerate}
\end{proof}
\begin{remark}
	Probabilistic representations as in the previous theorem exist for solutions when the corresponding boundaries are not $A$-natural (see Sections 4.6 and 5.1 in \cite{IM} ). However, the integral operator introduced in Theorem \ref{t:existenceNoNat} and used again in Corollary \ref{c:existenceNoNat} allows for an easy numerical method for the computation of solutions.
\end{remark}
\begin{remark}
	The similarities in the restrictions on the boundary conditions appearing in Table \ref{table} and those provided in McKean \cite{MKpar} are not surprising. Indeed, if $Y$ is the regular diffusion defined by the time change $Y_{A_t}=X_t$, it is easy to see that $g(X)\exp(-A)$ is a local martingale if and only if $g(Y)\exp(-I)$ is, where $I_t:=t$, provided that the Revuz measure $\mu_A$ of $A$  has a full support in $\elr$.  In this case, the monotone increasing and decreasing $g$ that make the former a local martingale coincide with the fundamental solutions of $\cL g =g$, where $\cL$  is the infinitesimal generator of $Y$.  If  $\mu_A$ does not have full support, the above equivalence will still hold if $X$ is killed at its first exit from the support of $\mu_A$. 
	
	Inspired by the above observation one may hope to  find the strictly positive subharmonic functions by solving the ordinary differential equation (ODE) associated with the infinitesimal generator of the time changed process $Y$, whose speed measure is given by $\mu_A$. However, by the same reasoning as above, this will work only if $\mu_A$ has full support. Otherwise, the fundamental solutions of the ODE will only define those subharmonic functions on  the support of $\mu_A$, { which will not be of much value, for instance, if $\mu_A$ is the Dirac measure associated with local time at a point.}   On the other hand, the methodology developed in this paper constructs these function on the whole $\elr$.
	
\end{remark}
Theorems \ref{t:existenceNoNat} and \ref{t:existenceNat} establish the existence of non-constant and monotone solutions of (\ref{e:choquet1}) and (\ref{e:choquet2}) under appropriate boundary conditions depending on the nature of the boundary. Denoting the nondecreasing (resp. non-increasing) and nonconstant solution of (\ref{e:choquet1}) (resp. (\ref{e:choquet2})) by $\psi_A$ (resp. $\phi_A$), Table \ref{table} summarises the boundary behaviour of these  functions - {\em fundamental solutions} using an ODE terminology -  in view of Propositions \ref{p:bcg1} and  \ref{p:kappas}, and the following Remark \ref{r:fundamental}.

\begin{table}[tb]
	\caption{Fundamental solutions}
	\label{table}\vspace{1mm}
	\par
	\begin{center}
		
		\bgroup
		\def\arraystretch{1.25}
		
		\begin{tabular}{lcccc}
			& $A$-regular & $A$-entrance & $A$-exit & $A$-natural \\
			\hline
			$\psi_A(\ell)$ & $\geq 0$ & $>0$ & $=0$&$=0$ \\
			\hline
			$\phi_A(\ell)$ & $<\infty$  & $=\infty$& $<\infty$ & $=\infty$ \\
			\hline
			$\frac{d^+\psi_A}{ds}(\ell)$ &$\geq0$ & $=0$  &  $>0$& $=0$ \\
			\hline
			$-\frac{d^+\phi_A}{ds}(\ell)$ & $<\infty$ & $<\infty$ & $=\infty$ & $=\infty$ \\
			\hline
		\end{tabular}
		
		\egroup
		
	\end{center}
	\par
	\begin{spacing}{1.0}
		\footnotesize  Table describes the boundary behaviour of $\psi$ and $\phi$ near $\ell$ depending on the boundary classification for $\ell$. In order to make sure that the fundamental solutions do not vanish completely the condition $\psi_A(\ell)+\frac{d^+\psi_A}{ds}(\ell)>0$ must be imposed when $\ell$ is $A$-regular. The behaviour of the fundamental solutions near the boundary $r$ can be found by exchanging $\ell$ with $r$ and $\phi$ with $\psi$ in the above table. 
	\end{spacing}\vspace{2mm}
\end{table}
\begin{remark} \label{r:fundamental} The behaviour of the nonincreasing solution $\phi_A$ near $\ell$ follows from (\ref{e:choquet2}) and the conditions defining the $A$-entrance, $A$-exit, etc.. Indeed, when $\ell$ is $A$-natural or $A$-exit, $\mu_A((\ell,b))=\infty$ for any $b \in \elr$. This implies in conjunction with (\ref{e:gsder}) that the derivative of $\phi_A$ must be infinite at $\ell$ since $\phi_A(\ell)>0$. The finiteness of $\mu_A((\ell,b))$ when $\ell$ is $A$-entrance, on the other hand, yields a finite derivative. 
	
	Similarly, when $\ell$ is $A$-entrance, $s(\ell)=-\infty$ that in turn implies $\phi_A(\ell)=\infty$ by looking at (\ref{e:choquet2}). If $\ell$ is $A$-exit, $s(\ell)>-\infty$ and $\mu_A(\ell,b)=\infty$, which yield $\phi_A(\ell)=\infty$ in view of (\ref{e:choquet2}). In the remaining case of $\ell$ being $A$-natural and $s(\ell)>-\infty$, that $\int_{\ell}^b (s(y)-s(\ell))\mu_A(dy)=\infty$ yields similarly $\phi_A(\ell)=\infty$.
\end{remark}

So far in this paper the focus has been on semi-uniformly integrable subharmonic functions. The next result -- akin to the representation of solutions of ODEs in terms of linearly independent solutions -- shows that this is enough to characterise all.
\begin{theorem} \label{t:representation}
	For any $g\in \cS^+$  there exists a PCAF $A$ with Revuz measure $\mu_A$ such that $g =\lambda_1 \psi_A+\lambda_2 \phi_A$, where $\psi_A$ and $\phi_A$ are nondecreasing and nonincreasing fundamental solutions from Table \ref{table}.
\end{theorem}
\begin{proof}
	Since $g$ is subharmonic, it is a convex function of $s$ and there exists a PCAF $B$ such that $g(X)-B$ is a $P^x$-local martingale for every $x\in \elr$. If $B\equiv 0$, $g$ must be an affine transformation of $s$, in which case $g$ is excessive and the claim holds with $\mu_A\equiv 0$. Note that if both $s(r)$ and $s(\ell)$ are infinite, that is $X$ is recurrent, only excessive functions are constants (see,.e.g., Exercise 10.39 in \cite{GTMP}). 
	
	Thus, suppose $B$ is not identically $0$. Since $g\in \cS^+$, $A_t =\int_0^t \frac{1}{g(X_s)}dB_s$ is well-defined as a PCAF. As observed before, $g(X)\exp(-A)$ can be easily checked to be a $P^x$-local martingale. 
	
	Let $\psi_A$ and $\phi_A$ are nondecreasing and nonincreasing fundamental solutions from Table \ref{table}.   Next consider an interval $(a^0,b^0)$ with $\ell <a^0<c<b^0<r$ and let $\lambda_1$ and $\lambda_2$ be such that
	\[
	\lambda_1 \psi_A{r}(a^0)+ \lambda_2 \phi_A(a^0)= g(a^0) \mbox{ and } \lambda_1 \psi_A(b^0) + \lambda_2 \phi_A(b^0)=g(b^0)
	\]
	noting that the above has a unique solution since $\phi_A$ and $\psi_A$ are linearly independent. Since $\exp(-A) \left\{g(X)- \lambda_1 \psi_A(X)-\lambda_2 \phi_A(X)\right\}$ is a $P^x$-local martingale and $g$ as well as $\phi_A$ and $\psi_A$  are continuous, one obtains
	\[
	g(x)-\lambda_1 \psi_A(x)-\lambda_2 \phi_A(x)=E^x\left[\exp(-A_{T_{ab}})\left\{g(X_{T_{ab}})- \lambda_1 \psi_A(X_{T_{ab}})-\lambda_2 \phi_A(X_{T_{ab}})\right\}\right]=0
	\]
	for any $x \in (a^0,b^0)$. Using the optional stopping theorem at $T_{zb^0}$ for $\ell<z <a^0$ and $x\in (a^0,b^0)$ shows    $g(z)=\lambda_1 \psi_A(z)+\lambda_2 \phi_A(z)$. Repeating the same argument at $T_{a^0z}$ for $z>b^0$ establishes $g(z)=\lambda_1 \psi_A(z)+\lambda_2 \phi_A(z)$ on $(b^0,r)$, hence the claim.
\end{proof}

\section{Path transformations via \Ito-Watanabe pairs} \label{s:transform}
This section is devoted to measure changes via \Ito-Watanabe pairs. Note that  the local martingale associated to a given \Ito-Watanabe pair  is not in general a uniformly integrable martingale, as can be seen from the following result in a special case.
\begin{proposition} \label{p:conv20}
	Suppose that $X$ is recurrent,  $f\geq 0$ and $A$ is a PCAF such that $f(X)\exp(-A)$ is a supermartingale.  Assume further that $f$ is continuous on $\elr$ and either $f(\ell+)$ or $f(r-)$ exist (with the possibility of being infinite). Then, $f(X_t)\exp(-A_t)\rar 0$, $P^x$-a.s. for all $x \in \elr$.
\end{proposition}
\begin{proof} Since $f(X)\exp(-A)$ is a non-negative supermartingale, it converges a.s.. If this limit is non-zero with non-zero $P^x$-probability, then $P^x(\lim_{t \rar \infty}f(X_t)=\infty)>0$ since $A_{\infty}=\infty$, a.s.. However, this implies $\lim_{t\rar \infty}X_t$ exists and equals $\ell$ or $r$ with positive probability, which contradicts recurrence.
\end{proof} 
Nevertheless,  one can still construct a Markov process, whose law is locally absolutely continuous with respect to that of the original process since $g(X)\exp(-A)$ is a {\em supermartingale multiplicative functional} (see Section 62 of \cite{GTMP}). 
\begin{theorem} \label{t:cofm}
	Consider an \Ito-Watanabe pair $(g,A)$, where $g$ is semi-uniformly integrable.   Then there exists a unique family of measures  $(Q^x)_{x \in \elr}$ on $(\Omega, \cF^u)$ rendering $X$ Markov with semigroup $(Q_t)_{t \geq 0}$ and $Q^x(X_0=x)=1$. Moreover, the following hold:
	\begin{enumerate}
		\item For every stopping time $T$ and $F\in \cF_T$
		\be \label{e:RN}
		Q^x(F,T<\zeta) =\frac{E^x[\chf_F \chf_{[T<\zeta]}g(X_T)\exp(-A_T)]}{g(x)} .
		\ee
		\item The semigroup $(Q_t)_{t \geq 0}$ coincides with that of a one-dimensional regular diffusion with no killing on $\elr$, scale function $s_g$ and speed measure $m_g$,
		where
		\[
		s_g(dx)=\frac{1}{g^2(x)}ds(x), \qquad m_g(dx)=g^2(x)m(dx).
		\]
		\item If $B$ is a PCAF of $X$ with Revuz measure $\mu$ under $(P^x)_{x\in \elr}$ and speed measure $m$, its Revuz measure under $(Q^x)_{x\in \elr}$ and speed measure $m_g$ is given by $\mu_g(dx) =g^2(x)\mu(dx)$. 
	\end{enumerate} 
	
\end{theorem}
\begin{proof}
	See the Appendix.
\end{proof}

\begin{remark} \label{r:cofm}
	A quick inspection of the proof reveals that Theorem \ref{t:cofm} remain valid if $g =c_1 g_1 + c_2 g_2$, where $c_i\geq 0$ and $(g_i,A)$ are \Ito-Watanabe pairs with semi-uniformly integrable $g_i$s. Thus, it is valid for all \Ito-Watanabe pairs in view of Theorem \ref{t:representation}.
\end{remark}
Remarkably  \Ito-Watanabe pairs transform recurrent diffusions to transient ones. 
\begin{corollary} \label{c:trtransform1}
	Suppose $X$ is recurrent and $A$ is a PCAF with $\mu_A(\bfE)>0$. Then $P^x(A_{\infty})=\infty)=1$. Consider $g=c_1 \psi_A + c_2 \phi_A$, where $\phi_A$ and $\psi_A$ are respectively the nonincreasing and nondecreasing functions defined in Table \ref{table}, $c_i\geq 0$ and $c_1+c_2>0$. Let $(Q^x)_{x\in \elr}$ denote the family of measures defined in Theorem \ref{t:cofm}. Then $X$ is transient under $(Q^x)_{x\in \elr}$. Moreover, 
	\begin{enumerate}
		\item If $c_1=0$, $Q^x(X_{\zeta -}=\ell)=1$.
		\item If $c_2=0$, $Q^x(X_{\zeta -}=r)=1$.
		\item If $c_1$ and $c_2$ are non-zero, $Q^x(X_{\zeta -}=r)>0$ and $Q^x(X_{\zeta -}=\ell)>0$.
	\end{enumerate}
\end{corollary}
\begin{proof} That $P^x(A_{\infty})=\infty)=1$ follows from Theorem \ref{t:Afinite}.
	
	Proof of only  the third remaining statement will be given  as the other cases are treated similarly. First observe from Table \ref{table} that $\psi_A(r)=\infty$ since $r$-cannot be $A$-regular or $A$-exit.  Similarly, $\phi_A(\ell)=\infty$.
	
	Suppose $ds(x)=dx$, without loss of generality, and note that
	\[
	s_g(\ell+)=-\int_{\ell}^c \frac{1}{g^2(x)}dx.
	\]
	Since $g$ is convex and $g(\ell+)=+\infty$, there exists $x^*<0$ and $k>0$ such that  $g(x)> -kx$ for all $x <x^*$. Thus, $s_g(\ell+)>-\infty$. Similarly, $s_g(r-)<\infty$. This proves the claim.
\end{proof}
The result above is a general case of the transient transformation considered in Proposition 5.1 in \cite{rectr}. A version  exists for transient diffusions as well.
\begin{corollary} \label{c:trtransform2}
	Suppose that $X$ is transient and $A$ is a PCAF with $\mu_A(\bfE)>0$. Consider $g=c_1 \psi_A + c_2 \phi_A$, where $\phi_A$ and $\psi_A$ are respectively the nonincreasing and nondecreasing functions defined in Table \ref{table}, $c_i\geq 0$ and $c_1+c_2>0$. Let $(Q^x)_{x\in \elr}$ denote the family of measures defined in Theorem \ref{t:cofm}. Then $X$ is transient under $(Q^x)_{x\in \elr}$. Moreover, $Q^x(X_{\zeta-}=\ell)>0$ if $c_2>0$ and $Q^x(X_{\zeta-}=r)>0$ if $c_1>0$.
\end{corollary}
\begin{proof}
	Suppose $ds(x)=dx$ without loss of generality. If $\ell >-\infty$, it is clear that $s_g(\ell+)$ is finite when $c_2>0$ since $\phi_A(\ell)>0$. Thus, suppose $\ell=-\infty$. Consequently, $\ell$ is either $A$-entrance or $A$-natural and $\phi_A(-\infty)=\infty$.    Since $\phi_A$ is convex, one has $\lim_{x\rar -\infty}\frac{\phi_A(x)}{-x}>0$, which in turn yields that $\int_{-\infty}^b\frac{1}{\phi_A^2(x)}dx<\infty$ and implies  the finiteness of $s_g(\ell+)$ when $c_2>0$.
	
	The implication of $c_1>0$ is proved similarly.
\end{proof}
The following example illustrates the above transformations in a rather simple setting, yet  exhibiting an intriguing singularity in the limit.
\begin{example}
	Consider a one-dimensional diffusion on natural scale with the state space $\bbR$. Let $\delta >0$ and note that $\mu_A(dx)=\frac{\epsilon_1(dx)}{\delta}$ is the Revuz measure for the PCAF $(2\delta)^{-1}L^1$, where $L^1$ is the semimartingale local time for $X$ at $1$. One can solve (\ref{e:IEh}) explicitly in this case to find
	\[
	\psi_A(x)= \delta +(x-1)^+ \mbox{ and } \phi_A(x)=\delta +(1-x)^+
	\]
	satisfying $\psi_A(1)=\phi_A(1)=\delta$. Moreover, $(\psi_A,A)$ and $(\phi_A,A)$ are \Ito-Watanabe pairs due to Theorem \ref{t:main1}.
	
	If one uses  $(\psi_A,A)$ to apply a path transformation to $X$ via Theorem \ref{t:cofm}, one obtains a transient diffusion (see Corollary \ref{c:trtransform1}) with scale function
	\[
	s_{\delta}(x)=\left\{\ba{ll}
	\frac{1}{\delta}-\frac{1}{\delta+x-1}, & \mbox{if } x\geq 1;\\
	\frac{x-1}{\delta^2}, & \mbox{if } x<1.
	\ea	\right.
	\]
	Note that $s_{\delta}(\infty)<\infty$, implying that the diffusion drifts towards infinity in the long run. Moreover, the potential density $u_{\delta}$ is given by
	\[
	u_{\delta}(x,y)=\frac{1}{\delta}-s_{\delta}(x\vee y).
	\]
	In particular if the original $X$ is a Brownian motion, the dynamics of $X$ under $(Q^x)$, where $Q^x$ is as defined in Theorem \ref{t:cofm}, is given by
	\[
	dX_t=dW_t+ \chf_{[X_t >1]}\frac{1}{\delta +X_t -1}dt,
	\]
	where $W$ is a standard Brownian motion. $X$ following the above dynamics still has the whole $\bbR$ as it state space. However, { by comparing to a 3-dimensional Bessel process}, it can be guessed that for smaller values of $\delta$ it must be getting harder for $X$ to move from the half space $(1,\infty)$ to $(-\infty,1)$ { due to the extremely large positive drift in the right neighbourhood of $1$}.  By taking formal limits as $\delta \rar 0$ one can see that $X$ is no longer regular: it is a Brownian motion on $(-\infty,1]$ while $X-1$ becomes a $3$-dimensional Bessel process on $[1,\infty)$. The set $\{1\}$ can be viewed as a {\em soft border} between two regimes allowing transitions from the Brownian regime to the Bessel one but not the other way around. 
	
	This formal description can be made more rigorous by analysing $L^1_{\infty}$ - the cumulative local time spent at $1$ at the lifetime. It is well-known that $L^1_{\infty}$ is exponentially distributed under $Q^1$ (see Paragraph 13 in Section II.2 in \cite{BorSal}). For a fixed $\delta>0$ the parameter of this exponential distribution equals $\frac{s'_{\delta}(1)}{2 u_{\delta}(1,1)}=\frac{1}{2\delta}$. In particular $Q^1(L^1_{\infty})=2\delta\rar  0$ as $\delta$ tends to $0$. Moreover, for $x <1<y$, $Q^x(T_y<\infty)=1$ whereas
	\[
	Q^y(T_x<\infty)=\frac{u_{\delta}(y,x)}{u_{\delta}(x,x)}=\frac{\delta^2}{(\delta +1-x)(\delta+y-1)}\rar 0 \mbox{ as } \delta \rar 0.
	\]
	Thus, the transitions from $(-\infty,1)$ to $[1,\infty)$ continue as $\delta$ gets small. However, once the border $\{1\}$ is reached, $X$ is strongly pulled into the interior of $[1,\infty)$ and finds it increasingly difficult to get back to the border.
\end{example}

\section{Conclusion} \label{s:conc}
Minimal subharmonic functions for a given one-dimensional diffusion are determined and used to develop a novel Choquet-type integral representation  for positive subharmonic functions. These minimal functions admit representation in terms of last passage times. The Choquet representation is used to construct integral equations for \Ito-Watanabe pairs in terms of the Revuz measure of the associated PCAF. Another novelty of the approach taken in this paper is that an easy numerical scheme exists for the solution of these equations when the Revuz measure of the PCAF satisfy an integrability condition. These pairs are then used to develop a theory of transient transformations. 
\appendix
\section{Appendix}
\begin{proof}[Proof of Proposition \ref{p:Choquniq}]
	The proof will be given for (\ref{e:Choquet1}) in case $s(\ell)>-\infty$ since the other cases can be handled the same way.
	
	Suppose that $g(\ell)=0$ and there are two representations given by $(\kappa_1, \mu_1)$ and $(\kappa_2,\mu_2)$. Define the CAFs $B^i$ by
	\[
	B^i_t=\int_{\ell}^r L^y_t\mu_i(dy).
	\]
	where $L^y$ is the diffusion local time. That is, $B^I$ is a PCAF with Revuz measure $\mu_i$.
	
	Next consider $\ell <a <x<b<r$  and observe that
	\[
	E^x[B^i_{T_{ab}}]=\int_{\ell}^r E^x[L^y_{T_{ab}}]\mu_i(dy),
	\]
	where $T_{ab}=T_a\wedge T_b$. Since $E^x[L^y_{T_{ab}}]=\frac{(s(x\wedge y)-s(a))(s(b)-s(x\vee y))}{s(b)-s(a)}$ for $y\in (a,b)$ (see Paragraph 13 on p.21 of \cite{BorSal}), one obtains
	\be \label{e:potBi}
	E^x[B^i_{T_{ab}}]=\int_{a}^b \frac{(s(x\wedge y)-s(a))(s(b)-s(x\vee y))}{s(b)-s(a)}\mu_i(dy).
	\ee
	
	On the other hand, 
	\bean
	E^x[g(X_{T_{ab}})]&=&\kappa_i (s(x)-s(\ell)) +\int_{\ell}^r E^x\left[\left(s(X_{T_{ab}})-s(y)\right)^+\right]\mu_i(dy)\\
	&=&\kappa_i (s(x)-s(\ell)) + \int_{\ell}^a \left\{(s(a)-s(y))\frac{s(b)-s(x)}{s(b)-s(a)}+(s(b)-s(y))\frac{s(x)-s(a)}{s(b)-s(a)}\right\}\mu_i(dy)\\
	&&+\int_{a+}^b (s(b)-s(y))\frac{s(x)-s(a)}{s(b)-s(a)}\mu_i(dy)\\
	&=&g(x)-\int_{\ell}^b (s(x)-s(y))^+\mu_i(dy)\\
	&&+\int_{\ell}^a \left\{(s(a)-s(y))\frac{s(b)-s(x)}{s(b)-s(a)}+(s(b)-s(y))\frac{s(x)-s(a)}{s(b)-s(a)}\right\}\mu_i(dy)\\
	&&+\int_{a+}^b (s(b)-s(y))\frac{s(x)-s(a)}{s(b)-s(a)}\mu_i(dy)\\
	&=&g(x)+ \int_{a+}^b\frac{(s(x)-s(a))(s(b)-s(y))-(s(x)-s(y))^+(s(b)-s(a))}{s(b)-s(a)}\mu_i(dy)\\
	&=&g(x)+\int_{a}^b \frac{(s(x\wedge y)-s(a))(s(b)-s(x\vee y))}{s(b)-s(a)}\mu_i(dy),
	\eean
	where the third equality follows from the representation of $g$, the fourth is due to the assumption that $x\in (a,b)$, and the last holds since the integrand equals $0$ when $y=a$.
	
	The above in conjunction with $(\ref{e:potBi})$ show that $E^x[B^1_{T_{ab}}]=E^x[B^2_{T_{ab}}]=E^x[g(X_{T_{ab}})]-g(x)$ for all $x \in (a,b)$. That is, the PCAF $B^i$s have the same potential when $X$ is stopped at $T_{ab}$. Therefore, $B^1$ and $B^2$ are indistinguishable upto $T_{ab}$ by Theorem IV.2.13 in \cite{BG}. Hence, $B^1=B^2$ by continuity and the arbitrariness of $a$ and $b$. Moreover, this implies they have the same Revuz measure in view of (\ref{d:Revmes}). Hence, $\mu_1=\mu_2$, which in turn implies $\kappa_1=\kappa_2$. The remaining assertions have already been proved. 
\end{proof}
\begin{proof}[Proof of Lemma \ref{l:mainR}] Proof will consider $v_c(x,y)=s(x \vee y)-s(c\vee y)$ and the other case is handled similarly.
	\begin{enumerate}[leftmargin=*]
		\item  Note that one can choose $c$ such that $a<c<x$ without loss of generality. Then,
		\bean
		E^x[g(X_{T_{ab}})]&=& g(c) + \kappa (s(x)-s(c)) \int_{\ell}^r (E^x [s(X_{T_{ab}}\vee y)]-s(c\vee y))g(y)\mu_A(dy) \\
		&=& g(c)+  \kappa (s(x)-s(c))+ \int_{\ell}^a (E^x s(X_{T_{ab}})-s(c))g(y)\mu_A(dy)\\
		&&+\int_{a+}^b (E^x s(X_{T_{ab}}\vee y)-s(c\vee y))g(y)\mu_A(dy)\\
		&=& g(c)+  \kappa (s(x)-s(c))+\int_{\ell}^{a} (s(x)-s(c))g(y)\mu_A(dy) \\
		&&+\int_{a+}^b (E^x [s(X_{T_{ab}}\vee y)]-s(c\vee y))g(y)\mu_A(dy),
		\eean
		where the first equality is due to the fact that $s(z\vee y)=s(z)$ for $y<a$ and $s(z\vee y)=s(y)$ whenever $z\in (a,b)$. On the other hand, for $y \in [a,b]$,
		\bean
		E^x[s(X_{T_{ab}}\vee y)]&=&s(y) P^x(T_a<T_b)+ s(b) P^x(T_b<T_a)= s(y)\frac{s(b)-s(x)}{s(b)-s(a)}+ s(b) \frac{s(x)-s(a)}{s(b)-s(a)}\\
		&=&s(x)+ \frac{(s(y)-s(a))(s(b)-s(x))}{s(b)-s(a)}.
		\eean
		Therefore,
		\bean
		E^x[g(X_{T_{ab}})]&=& g(c) + \kappa (s(x)-s(c))+\int_{\ell}^x (s(x\vee y)-s(c\vee y))g(y)\mu_A(dy)\nn \\
		&& +\int_{x+}^b \left(s(x)-s(y)-\frac{(s(y)-s(a))(s(b)-s(x))}{s(b)-s(a)}\right)g(y)\mu_A(dy)\nn\\
		&&\int_{a+}^x\frac{(s(y)-s(a))(s(b)-s(x))}{s(b)-s(a)}g(y)\mu_A(dy)\nn\\
		&=&g(x)+ \int_{x+}^b \frac{(s(x)-s(a))(s(b)-s(y))}{s(b)-s(a)}g(y)\mu_A(dy)\nn\\
		&&+\int_{a+}^x\frac{(s(y)-s(a))(s(b)-s(x))}{s(b)-s(a)}g(y)\mu_A(dy)\nn\\
		&=&g(x)+ \int_a^b \frac{(s(x\wedge y)-s(a))(s(b)-s(x\vee y))}{s(b)-s(a)}g(y)\mu_A(dy),
		\eean
		where the second equality is due to (\ref{e:IEh}), the third follows from that $y \mapsto \frac{(s(y)-s(a))(s(b)-s(x))}{s(b)-s(a)}$ vanishes at $y=a$.
		\item It is clear from the definition that $g$ is continuous on $\elr$. Thus, both $O^+$ and $O^-$ are open. Let $ x \in O^+$ and consider a neighborhood around $x$ with left endpoint $a$ and right endpoint $b$ such that $(a,b)\subset O^+$. Then, (\ref{e:EhabR}) yields
		\[
		E^x[g(X_{T_{ab}})]\geq g(x),
		\]
		as a consequence of the strict positivity of $g$ on $O^+$. This proves that $g$ is $s$-convex on $O^+$ since
		\[
		E^x[g(X_{T_{ab}})]= g(a)\frac{s(b)-s(x)}{s(b)-s(a)}+ g(b)\frac{s(x)-s(a)}{s(b)-s(a)}.
		\]
		The same technique can be used to prove $g$ is $s$-concave on $O^-$. 
		\item Suppose $v_c(x,y)= s(x\vee y)-s(c\vee y)$. First observe that the integrability assumption $\int_l^r |s(x\vee y)-s(c \vee y)||g(y)|\mu_A(dy)<\infty$ implies
		\[
		\int_l^{x} |g(y)|\mu_A(dy)<\infty
		\]
		for any $x \in \elr$. 
		
		Moreover, for any $a<x<b$,
		\[
		\frac{(s(x\wedge y)-s(a))(s(b)-s(x\vee y))}{s(b)-s(a)}\leq s(b)-s(x)
		\]
		for all $y \in (a,b)$. Thus, the dominated convergence theorem applied to (\ref{e:EhabR}) yields
		\[
		(s(b)-s(x))\lim_{a\rar \ell}\frac{g(a)}{s(b)-s(a)}+g(b)= \lim_{a\rar \ell}E^x[g(X_{T_{ab}})]=g(x)+\int_{\ell}^b u(b; x,y)g(y)\mu_A(dy),
		\]
		where $u(b;\cdot,\cdot)$ is given by 
		\be \label{e:ulb}
		u(b;x,y)=\lim_{a \rar \ell}\frac{\left(s(x\wedge y)-s(a)\right)\left(s(b)-s(x\vee y\right)}{s(b)-s(a)}.
		\ee
		
		Thus, since the right hand side is finite, one must have
		\[
		\lim_{a\rar \ell}\frac{g(a)}{s(b)-s(a)}<\infty,
		\]	
		which in turn yields the finiteness of $g(\ell+)$ if $s(\ell)>-\infty$. 
		
		Note in particular that if $g$ is u.i. near $\ell$ and $s(\ell)=-\infty$, $\lim_{a\rar \ell}E^x[g(X_{T_{ab}})]=g(b)$ and, consequently, $\lim_{a\rar \ell}\frac{g(a)}{s(b)-s(a)}=0$. 
		
		Thus,
		\[
		\lim_{a\rar \ell}\frac{g(a)}{s(b)-s(a)}=\lim_{x\rar \ell}\frac{g(x)}{s(b)-s(x)}+\lim_{x\rar \ell}\int_{\ell}^b \frac{s(b)-s(x\vee y)}{s(b)-s(x)}g(y)\mu_A(dy),
		\]
		which in turn yields 
		\[
		\lim_{x\rar \ell}\int_{\ell}^b \frac{s(b)-s(x\vee y)}{s(b)-s(x)}g(y)\mu_A(dy)=0.
		\]
		
		On the other hand, (\ref{e:genelIE}) implies
		\be \label{e:gblim}
		0=\lim_{x\rar \ell}\frac{g(x)}{s(b)-s(x)}+ \kappa+ \lim_{x\rar \ell}\int_{\ell}^b \frac{s(b)-s(x\vee y)}{s(b)-s(x)}g(y)\mu_A(dy),
		\ee
		which establishes $\kappa=-\lim_{x\rar \ell}\frac{g(x)}{s(b)-s(x)}=0$. Consequently, $g(\ell+)$ is finite.
		\item If $g$ changes its sign,  there exists a $c^* \in (\ell,r)$ such that either $g$ is decreasing, $s$-convex on $(\ell,c^*)$ and $s$-concave on $(c^*,r)$ or increasing, $s$-concave on $(\ell,c^*)$ and $s$-convex on $(c^*,r)$. Since $-g$ also solves (\ref{e:genelIE}), assume without loss of generality that the former case holds. Fix $c \in (\ell,c^*)$ and let $x \in (c,c^*)$ be arbitrary. Then, assuming $v_c(x,y)= s(x\vee y)-s(c\vee y)$ 
		\[
		g(x)= g(c) +  \int_{\ell}^x( s(x)-s(c\vee y))g(y)\mu_A(dy) \geq g(c)
		\]
		since $g$ is non-negative on $(\ell,c^*)$. This shows $g$ is increasing on $(l,c^*)$ yielding a contradiction. 
		
		Similarly, if $v_c(x,y)=s(c\wedge y)-s(x\wedge y)$, let $c^*<c<x$ and note that $g$ is nonpositive on $(c^*,r)$. Then,
		\[
		g(x)= g(c) +  \int_{x+}^r( s(c)-s(x \wedge y))g(y)\mu_A(dy) \geq g(c)
		\]
		contradicts that $g$ is decreasing.
		\item  Note that, for sufficiently small $h>0$ and  $x\in [\ell,r)$ such that $g(x)<\infty$,
		\bean
		g(x+h)-g(x)&=&\kappa (s(x+h)-s(x))+(s((x+h))-s(x))\int_{\ell}^x g(y)\mu_A(dy)\\
		&&+\int_{x+}^{x+h}(s(x+h)-s(y))g(y)\mu_A(dy),
		\eean
		which in turn yields 
		\[
		\frac{d^+g(x)}{ds}=\kappa+ \int_{\ell}^x g(y)\mu_A(dy)
		\]
		since $g$ is continuous, $\mu_A$ is finite on any small neighbourhood around $x$ and does not charge $\{\ell\}$. 
		
		Similarly,
		\[
		g(x)-g(x-h)=\kappa(s(x)-s(x-h)) +(s(x)-s(x-h))\int_{\ell}^{x-h}g(y)\mu_A(dy) + \int_{(x-h)+}^{x}(s(x)-s(y))g(y)\mu_A(dy),
		\]
		and therefore 
		\[
		\frac{d^-g(x)}{ds}=\kappa + \int_{\ell}^{x-} g(y)\mu_A(dy).
		\]
		The other case for $v_c$ is handled in the same manner.
	\end{enumerate}
\end{proof}

\begin{proof}[Proof of Theorem \ref{t:main1}]
	
	\Implies{t:eq1}{t:eq2}:  Since proofs are analogous, assume $g\in G_{\ell}$. Since $g\in \cS$, it admits a representation given by (\ref{e:Choquetgen}).  Let $c^*$ be a point where $g(c^*)=\inf_{x\in \bfE}g(x)$ and note that $c^*$ could equal $\ell$. 
	
	First consider the case $c^*\in \elr$. It follows from Theorem \ref{t:choquet_gen} the existence of two Borel measures $\mu_1$ and $\mu_2$ on the Borel subsets of $\elr$ such that
	\be \label{e:galt1}
	g(x)= g(c^*)+   \int_{\ell}^r (s(x)-s(y))^+ \mu_1(dy) +\int_{\ell}^r (s(y)-s(x))^+ \mu_2(dy)
	\ee
	such that $\mu_1((\ell, c^*))=\mu_2((c^*,r))=0$.  Moreover, since $g(X)-\int_0^{\cdot}g(X_{t-})dA_t$ is a local martingale, $\mu_1+ \mu_2 = g\cdot \mu_A$. 
	
	Observe that the above  implies $\int_{\ell}^{c^*}\mu_2(dy)<\infty$. Indeed, straightforward calculations yield
	\[
	\frac{dg^+}{ds}(x)=-\int_{x+}^r\mu_2(dy)+\int_{\ell}^x\mu_1(dy).
	\]
	In particular, $\infty>\frac{dg^+}{ds}(\ell)=-\mu_2((\ell,c^*])$.
	
	Next, rewriting (\ref{e:galt1}) one arrives at
	\bean
	g(x)&=&g(c^*) -(s(x)-s(c^*))\int_{\ell}^{c^*}\mu_2(dy)+ \int_{\ell}^r (s(x)-s(y))^+ \mu_1(dy) \\
	&&+\int_{\ell}^r\left\{(s(y)-s(x))^+ + s(x)-s(c)\right\} \mu_2(dy)\\
	&=&g(c^*) -(s(x)-s(c^*))\mu_2((\ell,c^*])+ \int_{\ell}^r (s(x\vee c^*)-s(y\vee c^*)) \mu_1(dy) \\
	&&+\int_{\ell}^r\left\{\chf_{y\geq x}(s(y)-s(c^*)) + \chf_{y<x}(s(x)-s(c^*))\right\} \mu_2(dy)\\
	&=&g(c^*)-(s(x)-s(c^*))\mu_2((\ell,c^*])+ \int_{\ell}^r (s(x\vee c^*)-s(y\vee c^*)) \mu_1(dy) \\
	&&+\int_{\ell}^r(s(x\vee c^*)-s(y\vee c^*)) \mu_2(dy)\\
	&=&g(c^*)-(s(x)-s(c^*))\mu_2((\ell,c^*])+  \int_{\ell}^r (s(x\vee c^*)-s(y\vee c^*)) g(y)\mu_A(dy).
	\eean
	
	Thus, (\ref{e:IEh}) holds with $\kappa_1=-\mu_2((\ell,c^*])$ and $c=c^*$.  Now a quick inspection of (\ref{e:IEh}) reveals that if $g$ solves  (\ref{e:IEh}) for one $c$, it solves it for all. 
	
	On the other hand, if $c=\ell$,  $g$ is non-decreasing with decomposition  
	\[
	g(x)=g(\ell+)+  \int_{\ell}^r (s(x)-s(y))^+ g(y)\mu_A(dy),
	\]
	utilising the fact that $\kappa$ in (\ref{e:Choquetgen1}) coincides with $\frac{d^+g(\ell+)}{ds}$ and the Revuz measure of the PCAF $B$ that makes $g(X)-B$ a local martingale is given by $g\cdot \mu_A$ as before.  Thus, repeating the calculations leading to (\ref{e:ginteq}) yields that (\ref{e:IEh}) holds for any $c\in \elr$.
	
	\Implies{t:eq2}{t:eq1} The proof will be given for the first equation as the other case can be done similarly. 
	
	It follows from Lemma \ref{l:mainR} that $g$ is $s$-convex on $\elr$ and its right $s$-derivative at $\ell$ is finite. In particular, $g$ is subharmonic and there exists a PCAF $B$ by Theorem 51.7 in \cite{GTMP} that $g(X)-B$ is a $P^x$-local martingale for any $x \in \elr$. In particular, for any $\ell <a<x<b<r$,
	\[
	E^x[g(X_{T_{ab}})]=g(x)+ E^x[B_{T_{ab}}]=g(x)+ \int_a^b \frac{(s(x\wedge y)-s(a))(s(b)-s(x\vee y))}{s(b)-s(a)}\mu_B(dy)
	\]
	due to (\ref{e:potentialA}), where $\mu_B$ is the Revuz measure associated with $B$. On the other hand, (\ref{e:EhabR}) yields
	\[
	E^x[g(X_{T_{ab}})]=g(x)+  \int_a^b \frac{(s(x\wedge y)-s(a))(s(b)-s(x\vee y))}{s(b)-s(a)}g(y)\mu_A(dy).
	\]
	Since 
	\[
	\int_a^b \frac{(s(x\wedge y)-s(a))(s(b)-s(x\vee y))}{s(b)-s(a)}g(y)\mu_A(dy)=E^x\int_0^{T_{ab}} g(X_t)dA_t,
	\]
	one deduces easily that $E^x[B_{T_{ab}}]=E^x\int_0^{T_{ab}} g(X_t)dA_t$ for all $a<x<b$. That is, the potentials of $g \cdot A$ and $B$ coincide when $X$ is killed at $T_{ab}$, which in turn leads to the fact that $B$ and $g \cdot A$  are indistinguishable by Theorem IV.2.13 in \cite{BG} since $a$ and $b$ are arbitrary. Thus, $g(X)-\int_0^{\cdot}g(X_t)dA_t$ is a local martingale. A simple integration by parts and the fact that $g$ is bounded on the compact intervals of $\elr$ show that $g(X)\exp(-A)$ is a local martingale. Since $g$ is not identically $0$, $(g,A)$ is an \Ito-Watanabe pair in view of Theorem \ref{t:submvanish}.
	
	Uniform integrability near $\ell$ is obvious since $g$ is bounded on $(\ell, b)$ for any $b<r$ in view of Lemma \ref{l:mainR} and the fact that $g\geq 0$. 
\end{proof}

\begin{proof}[Proof of Theorem \ref{t:cofm}]
	The first statement follows directly from Theorem 62.19 in \cite{GTMP}.
	
	To prove the second statement observe that the killing measure on $\elr$ under $Q^x$ is null since there is no killing under $P^x$ and $g(X)\exp(-A)$ is a $P^x$-martingale when stopped at $T_{ab}$ for any $\ell<a<b<r$. 
	
	Moreover, $m_g$ is a symmetry measure for $(Q_t)$. Indeed, if $f$ and $h$ are bounded and measurable functions vanishing at $\Delta$, then
	\bean
	\int_{\ell}^r Q^x[f(X_t)]h(x)g^2(x)m(dx)&=&\int_{\ell}^r E^x[f(X_t)g(X_t)\exp(-A_t)]h(x)g(x)m(dx)\\
	&=&\int_{\ell}^r E^x[h(X_t)g(X_t)\exp(-A_t)]f(x)g(x)m(dx)\\
	&=&\int_{\ell}^r Q^x[h(X_t)]f(x)g^2(x)m(dx),
	\eean
	where the second equality follows from the fact that $m$ is the symmetry measure for $(P_t)$ and $\exp(-A)$ is a multiplicative functional in view of Theorem 13.25 in \cite{ChungWalsh}. Thus, $m_g$ is a speed measure associated to $(Q_t)_{t \geq 0}$. 
	
	Next let us observe that $s_g(X)$ is a $Q^x$-local martingale, where $s_g(x) =\int_c^x s_g(dx)$ for an arbitrary $c \in \elr$. However, this is equivalent to $s_g(X)g(X)\exp(-A)$ is a $P^x$-local martingale. That is, $(s_g g, A)$ has to be an \Ito-Watanabe pair. By killing $X$ at $T_a$ if necessary, this will follow from Corollary \ref{c:main1} if $s_g g$ solves (\ref{e:IEh_choquet}) on $(a,r)$ for any $a>\ell$ once $\ell$ is replaced by $a$.
	
	Indeed, redefining $s_g$ so that $s_g(a)=0$ one has via 
	\[
	\frac{d^+ s_g g}{ds}=\frac{1}{g} + s_g\frac{d^+g}{ds}
	\]
	and integration by parts that 
	\bean
	\frac{d^+ s_g g}{ds}(x)&=&\frac{1}{g(x)}+ s_g(x) \left(\frac{d^+g(a)}{ds}+ \int_a^xg(y)\mu_A(dy)\right)\\
	&=&\frac{1}{g(x)}+ \int_a^xs_g(y)g(y)\mu_A(dy)+ \int_a^x\frac{d^+g(y)}{ds}g^{-2}(y)\mu_A(dy)\\
	&=&\frac{1}{g(a)}+ \int_a^x s_g(y)g(y)\mu_A(dy),
	\eean
	where the first equality follows from (\ref{e:gsder}). Therefore,
	\bean
	s_g(x)g(x)&=&\frac{1}{g(a)}(s(x)-s(a))+ \int_a^x \int_a^z s_g(y)g(y)\mu_A(dy)ds(z)\\
	&=&\frac{1}{g(a)}(s(x)-s(a))+ \int_a^x (s(x)-s(y))s_g(y)g(y)\mu_A(dy),
	\eean
	which establishes that $(s_g g, A)$ is an \Ito-Watanabe pair in view of Corollary \ref{c:main1}.
	
	Therefore,  once the speed measure $m_g$ is fixed, the associated scale function, $s^*$, will satisfy $s^*(dx)=k s_g(dx)$ for some $k>0$. Thus, the proof will be complete once it is shown that $k=1$. To this end, note that the potential density  of $X$ killed at $T_{ab}$ under the dynamics defined by $(Q_t)$ is given by $ku^*_{ab}$, where
	\[
	u^*_{ab}(x,y)= \frac{(s_g(x \wedge y)-s_g(a))(s_g(b)-s_g(x\vee y))}{s_g(b)-s_g(a)}. 
	\]
	To determine $k$, the quantity $Q^x(s(X_{T_{ab}}))-s(x)$ will be computed in two ways. First,
	\be \label{Qxs1}
	Q^x(s(X_{T_{ab}}))-s(x)= s(a)\frac{s_g(b)-s_g(x)}{s_g(b)-s_g(a)}+s(b)\frac{s_g(x)-s_g(a)}{s_g(b)-s_g(a)}-s(x).
	\ee
	On the other hand,
	\bean
	g(x) Q^x(s(X_{T_{ab}}))&=& E^x \left[s(X_{T_{ab}})g(X_{T_{ab}})\exp(-A_{T_{ab}})\right] \\
	&=& g(x)s(x) + E^x\left[\int_0^{T_{ab}}\exp(-A_t)g'(X_t)dB_t\right]\\
	&=&g(x)s(x) + g(x)Q^x\left[\int_0^{T_{ab}}\frac{g'(X_t)}{g(X_t)}dB_t\right],
	\eean
	where $B$ is as in Theorem \ref{t:ibp} and $g'$ stands for the left derivative of $g$ with respect to $s$. Since the Revuz measure of $B$ under $Q^x$ becomes $2g^2(x)s(dx)$ as will be shown below, one obtains
	\be \label{Qxs2}
	Q^x(s(X_{T_{ab}}))-s(x)=2k\int_a^b u_{ab}^*(x,y) g'(y) g(y)s(dy).
	\ee
	Now, combining (\ref{Qxs1}) and (\ref{Qxs2}) and repeating the similar calculations used in the proof of Theorem VII.3.12 in \cite{RY} yield
	\[
	\frac{ds}{ds_g}(x)-\frac{ds}{ds_g}(y)=2k\int_x^y g'(y)g(y)s(dy).
	\]
	However, the left hand side of the above is $g^2(x)-g^2(y)$ while the right hand side equals $k(g^2(x)-g^2(y))$. Thus, $k$ must equal $1$. 
	
	Thus, it remains to prove the last statement. First, suppose $B_t:=\int_0^t f(X_s)ds$ for a non-negative measurable $f$. Then,
	\[
	Q^x(B_{T_{ab}})=\int_a^b u^*_{ab}(x,y)f(y) m_g(dy)=\int_a^b u^*_{ab}(x,y)f(y) g^2(y)m(dy),
	\]
	for any $\ell <a<b<r$, which implies the Revuz measure under  $(Q^x)$ given the speed measure $m_g$ equals $f(y)g^2(y)m(dy)$. Since the corresponding measure under $(P^x)$ is given by $f(y)m(dy)$, the claim follows for all such $B$. 
	
	Moreover, by the occupation times formula $B_{T_{ab}}=\int_{\ell}^r L^y_{T_{ab}} f(y) m(dy)$, where $L^y$ is the local time of $X$ at level $y$ under $(P^x)$ with respect to $m$. Thus, for any non-negative measurable $f$
	\[
	\int_a^b u^*_{ab}(x,y)f(y) g^2(y)m(dy)=\int_{\ell}^r Q^x(L^y_{T_{ab}}) f(y) m(dy).
	\]
	On the other hand, $L^y$ is a PCAF for $X$ under $(Q^x)$ and its support is contained in $\{y\}$ since $Q^x \ll P^x$ on $\cF_t^*$  for every $t$ when restricted to $[t<\zeta]$. Then, by Proposition 68.1 in \cite{GTMP} $L^y$ is proportional to the local time at $y$ with respect to $m_g$ under $Q^x$. Therefore, $Q^x(L^y_{T_{ab}})= \alpha u^*_{ab}(x,y)$ for some $\alpha>0$, which can be easily seen equal to $g^2(y)$ in view of the above. This in turn implies the Revuz measure for $L^y$ is given by $g^2(y)\epsilon_y(dx)$, where $\epsilon_y$ is the Dirac measure at $y$. The proof is now complete since if $B$ is a PCAF with  Revuz measure $\mu$ under  $(P^x)$ for the speed measure $m$, $B=\int_{\ell}^r\mu(dy)  L^y $.
\end{proof}
\bibliographystyle{siam}
\bibliography{../ref.bib}
\end{document}